\documentclass{commat}

\usepackage[stable]{footmisc}
\usepackage{mathrsfs}

\DeclareMathOperator\ad{ad}
\DeclareMathOperator\Ad{Ad}
\DeclareMathOperator\AD{AD}
\DeclareMathOperator\Cong{Cong}
\DeclareMathOperator\Diag{Diag}
\DeclareMathOperator\End{End}
\DeclareMathOperator\tr{tr}
\DeclareMathOperator\wk{weak}

\newcommand\Acal{\mathcal{A}}
\newcommand\be{\begin{equation}}
\newcommand\bpmtx{\begin{pmatrix}}
\newcommand\Ccal{\mathcal{C}}
\newcommand\Dref{Definition~\ref}
\newcommand\ee{\end{equation}}
\newcommand\epmtx{\end{pmatrix}}
\newcommand\Eref{Example~\ref}
\newcommand\eroman{\etype{\roman}}
\newcommand\Lcal{\mathcal{L}}
\newcommand\Lie{reversible Lie }
\newcommand\lra{\longrightarrow}
\newcommand\Lref{Lemma~\ref}
\newcommand\mcA{\mathcal{A}}
\newcommand\mcI{\mathcal{I}}
\newcommand\mcL{\mathcal{L}}
\newcommand\mcM{\mathcal{M}}
\newcommand\mcN{\mathcal{N}}
\newcommand\NN{\mathbb{N}}
\newcommand\QQ{\mathbb{Q}}
\newcommand\Rcal{\mathcal{R}}
\newcommand\Rref{Remark~\ref}
\newcommand\RR{\mathbb{R}}

\newcommand\Ucal{\mathcal{U}}
\newcommand\vep{\varepsilon}
\newcommand\weakLie{quasi Lie}
\newcommand{\etype}[1]{\renewcommand{\labelenumi}{(#1{enumi})}}
\newcommand{\one}{1}
\newcommand{\tT}{\mathcal{T}}
\newcommand{\zero}{0}
\newcommand{\Z}{\mathbb{Z}}

\title{%
   Lie pairs
    }

\author{%
    Letterio Gatto and Louis Rowen
    }

\authorinfo[L. Gatto]{Dipartimento di Scienze Matematiche Politecnico di
Torino, Italy}{letterio.gatto@polito.it}

\authorinfo[L. Rowen]{Bar-Ilan University, Israel}{rowen@math.biu.ac.il}

\abstract{%
    Extending the theory of systems, we   introduce a theory of Lie
semialgebra ``pairs'' which parallels the classical theory of Lie algebras, but with a ``null set'' replacing $0$. A selection of examples is given. These Lie pairs comprise two
categories in addition to the universal algebraic definition, one with ``weak Lie morphisms'' preserving null sums, and the other with ``$\preceq$-morphisms''  preserving   a surpassing relation $\preceq$ that replaces equality.  We provide versions of the PBW (Poincar\'{e}-Birkhoff-Witt) Theorem in these categories.
    }



\VOLUME{32}
\YEAR{2024}
\NUMBER{2}
\firstpage{71}
\DOI{https://doi.org/10.46298/cm.12413}

\begin{document}

\tableofcontents

\section{Introduction}
The purpose of this paper is to take a further step  towards   a general flexible framework for a unified  treatment of classical algebraic structures together with those arising in a tropical  context, where typically one cannot rely on the existence of an additive inverse (e.g., as in the celebrated max-plus algebra). The present research includes    Lie algebras in this general picture, in the sense that we are about to explain. In other words, it may be considered as one more stage of a wider program initiated
some years ago by the second author, through the theory of {\em triples} and {\em systems} (see e.g. \cite{Row16,Row17}), which has already proved  successful in
revisiting classical algebraic phenomena by embedding them in a tropical context. Among its applications, we recall the construction of an
effective tropical substitute of the exterior algebra, along with a natural extension of the
Cayley-Hamilton  theorem \cite{GaR} for endomorphisms of modules over semialgebras.

The simple but effective idea for remedying the lack
of negation is to introduce an endomorphism $(-)$, whose square is the
identity,  to which is attached a {\em surpassing relation}. A further step was taken  in \cite{ChaGaRo}, where a theory of Clifford semialgebra is proposed. In more traditional contexts, Clifford algebras are examples of Lie super-algebras, so
\cite{ChaGaRo} may be considered as the
first relevant example  of Lie semi-(super)algebras obtained within the
already collocated framework of triples and systems. It was applied to extend to the tropical framework  the  polynomial representation of Lie algebras of endomorphisms of a vector space, in the same spirit of \cite{DJKM}.

Meanwhile, a theory of semialgebra pairs has been  developed  in \cite{AGR2,JMR}, with the aim of exploiting, through their axiomatization, the formal properties enjoyed by the surpassing
relation associated to a negation map,  and eventually expunging the latter. For the
reader's convenience, we recall here that many of the classical key properties of
several  algebraic structures in traditional frameworks are recovered by the formalism of {\em systems}, in which
equality is replaced by the surpassing relation.

This  premise should make clear that there is ample motivation to cope with
the more tricky situation provided by the tropical version of Lie algebras; moreover, in
view of \cite{JMR}, it is natural to investigate and to set the foundation of a theory
of Lie pairs, generalizing Lie (semi)-algebras. These are  pairs $(\Lcal, \Lcal_0)$ of modules over some commutative semiring
$C$, endowed with a product $[\,\,]:\Lcal\times\Lcal\lra \Lcal$, $
(x,y)\mapsto [xy]$, satisfying suitable properties inspired by the classical Lie theory, and for which the skew-symmetric and ``Jacobi identity'' features of the
theory are all subsumed in the submodule $\Lcal_0$ of $\Lcal$, which basically contains all of the
relevant relations.

To show the reader quickly what we are
talking about, the Lie  bracket $[\,\,]$ satisfies the property $[xx]\in \Lcal_0$ for all
$x\in\Lcal$. Moreover, one naturally requires $[xy]+[yx]$ to lie in $\Lcal_0$, regardless
of the choice of $(x,y)\in \Lcal\times\Lcal$. We thereby define, in  case of free
modules over a base semiring, the structure constants of a Lie pair. The attractiveness of the theory comes from the freedom in defining $\Lcal_0$ as the basket containing all the undesired  appurtenances (due to
skew-symmetry or the Jacobi relation) occurring in the formal manipulations, which enhance the ability to construct families of examples of Lie pairs.
 It is  also important to stress
that the proposed axiomatization is natural, and one recovers the Lie semialgebras in
the sense that when $\Lcal$ is a module over a commutative ring and
when $\Lcal_0=\{0\}$ one obtains the classical definition of Lie algebras, and all of our examples work in this case, and reproduce the classical ones, like, e.g., the cross product.

Although our take is more along traditional structural algebraic lines,
following Jacobson~\cite{Jac} and Humphreys~\cite{H}, but relying on the subset $\Lcal_0$
taking the place of $\zero,$
it should be remarked that the literature has already seen research aimed to build theories of Lie semialgebras, for instance in the work
%
%
by
Hilgert and Hofmann \cite{HH}, relying on the Campbell-Hausdorff formula.

As remarked, this theory of
``pairs'' is an
outgrowth of   ``triples''  and
``systems,'' cf.~\cite{GaR, AGR2, JuR1, Row16, Row17}, which have unified classical algebraic
theory with tropical theory and other examples
including hyperfields, as explained in~\cite{AGR1}. Pairs are used in linear algebra
in \cite{AGR2}, and in generalizing commutative algebra theory in~\cite{AGR2}.
But whereas the
set $\tT$ of ``tangible'' elements  (that is the elements of the ground set)  played a crucial role
in  semiring and hyperring systems, in this study of ``Lie pairs'' we do not
deal with tangible elements at all. In other words, $\tT = \emptyset.$

We  bring in  a ``surpassing relation'' in \S\ref{ddi}, to be preserved by ``$\preceq$-morphisms'' in its appropriate category.
There are three possible categories, corresponding to the three versions of morphisms given in
\Dref{symsyst} and \Dref{precmor}. The ``weak  morphisms'' and ``$\preceq$-morphisms'' are inspired by the theory of
hyperfields, cf.~\cite{Vir}.

Among the main thrusts of this paper is to lay out the categorical foundations of Lie pairs in
\Dref{LieD}, paying attention to examples inspired by the classical theory, obtaining categories
parallel to \cite{AGR1,AGR2}.     At times  negation can be replaced by a ``pre-weak negation map'' $\psi   $ satisfying $x+\psi( x)\in \mcL_0$,  cf.~Theorem~\ref{pLie}. We also
introduce pre-Lie $\vep$-pairs, the  analog of pre-Lie algebras, in \Dref{ass}, and show how to
obtain a Lie pair from a pre-Lie $\vep$-pairs in Theorem~\ref{preLir}.
The Lie versions of morphisms are given in  \Dref{morp}.  It might seem strange that there are three different versions of Lie morphisms, but this also happens in other non-classical algebraic theories such as hyperfields \cite{Vir}.
Our main  category uses ``weak Lie morphisms,''  with many natural examples provided along
the way.

 We extend   major examples from classical Lie theory, to be described shortly.  On the other hand, there is a Lie
 version of Krasner's hyperfield construction of \cite{krasner}, given in~\S\ref{Kras0}.

 To test the viability of these notions, we prove versions of the PBW (Poincare-Birkhoff-Witt) Theorem in these three categories (Theorems~\ref{PBW0},  ~\ref{thm:PBW}, and \ref{PBW3}).

\subsection{Shape of the paper}
To help the reader to get oriented in the exposition of so many new, though natural, notions, we now give   a glimpse of how the paper is organized, also to share the feeling of what is in it.

To ease the reading, and to make the paper as self contained as possible, we collect in section \ref{prelim} all the prerequisites and notation to be used  in the article. The framework is very general, which explains why we put so much emphasis on very sparse algebraic structures like magmas and bimagmas. Pairs and negation maps are quickly recalled in Section \ref{subsec:pairs} and \ref{negmap}.   Weak Property N, introduced in Section \ref{subsec:N}, is necessary because we cannot expect, as  easily seen by basic examples, for nontrivial negation maps to exist.

The theoretical core of the paper is Section \ref{LieD0}, where we collect foundational material about the theory of Lie pairs in our sense, the basic morphisms used in the rest of the paper.

To show   that our theory is not empty we devote Section \ref{subsec:MLC}  to major Lie constructions (such as a Lie pair from an associative pair in Theorem~\ref{pLie}, and  from an associative pair with involution in Theorems~\ref{Linv} and \ref{Linv1}) and examples (the classical constructions of Theorem~\ref{clas1}, and low dimensional examples in \S\ref{subsec:ME} including the cross product), as well as Filiform pairs in \S\ref{fili} and an example motivated by hyperfield theory in Theorem~\ref{Kras}.

One standard technique for working with semialgebras,   to cope with the unavailability of additive inverses, is that of doubling, which in a sense recalls the construction of the integers from the natural numbers, but where we avoid   taking the quotient modulo a congruence.

Among the most natural examples of Lie pairs,   is  one where the Lie bracket is obtained as the Lie commutator in an associative semialgebra. The construction is straightforward. The standard model of any associative (semi)-algebra is that of a quotient of the tensor (semi)-algebra associated to a module. This is why in Section \ref{sec:tsflp}  we  treat   tensor semialgebras of free Lie pairs.

In \S\ref{PBW} we address the natural question: given any Lie pair $(\Lcal, \Lcal_0)$, can we construct an associative pair $(\Acal, \Acal_0)$ in which  $(\Lcal,\Lcal_0)$ is embedded, in such a way that the commutator restricts to the given Lie bracket? This would be the extension of the Poincar\'e--Birkhoff--Witt (PBW) theorem in our context. In our concluding subsection \ref{subsec:PBWy} we analyze the corresponding PBW situation in the various versions of Lie pairs. We shall see that the construction is unambiguous for each of the  versions considered, although it must take into account the corresponding category.

    \section{Preliminaries and Notation}\label{prelim}

First we review some definitions from \cite{JMR}.
As usual we denote as $\NN$  the semiring of  natural numbers
(including $0$), and $\NN^+:= \NN \setminus \{0\}.$
\begin{definition}\label{basicdefs}$ $
\begin{enumerate}\eroman
\item
A \textbf{magma} is a set  with a binary operation denoted $(+)$
(addition) or $(\cdot)$ (multiplication). At times we also require a
neutral element, written as $\zero$ or $\one$ respectively.

A \textbf{semigroup} is a magma   whose given binary operation
satisfies the law of associativity.

A \textbf{bimagma} $\mathcal A$ is a multiplicative monoid $(\mathcal A,\cdot,\one)$ which also is an additive semigroup $(\mathcal A,+,\zero)$, satisfying
$\zero b = b \zero = \zero$ for all $b \in \mathcal A$. (Thus, for us, bimagmas are associative both for multiplication and addition.)

A \textbf{d-bimagma} is a bimagma which is \textbf{distributive}, by which we mean $$\left(\sum _i x_i\right)\left(\sum _j y_j\right) = \sum _{i,j} x_i y_j,\quad \mbox{for all} \ x_i,y_j \in \mcA. $$

A \textbf{semiring}
(cf.~\cite{Cos},\cite{golan92}) $(\mathcal A, +, \cdot, \zero, \one)$ is  a (multiplicatively) associative d-bimagma also with a multiplicative identity $\one.$ A \textbf{semifield} is a
semiring in which every nonzero element is invertible.

\item $\Ccal$ always will denote a commutative semiring, e.g.~$\NN$ or $\QQ_{+}$ or the max-plus algebra.  Often $\Ccal$ will be a semifield.

\item
A (left) $\mcA$-\textbf{module}\footnote{For convenience, we are defining modules with $\zero$, but since we do not require negation, the $\zero$ element could be dispensed with.} over a semigroup  $C$
 is a semigroup $( \mathcal M,+,\zero)$  endowed with scalar
multiplication $C\times \mathcal M \to \mathcal M$ satisfying the
following axioms, for all $c, c_i\in C$ and $y, y_i \in \mathcal M$:

\begin{enumerate}\eroman
\item
$c \zero=\zero c = \zero,$ (i.e., $\zero$ is absorbing).
\item $c\sum y_i = \sum cy_i,$ $(\sum c_i)y = \sum (c_iy).$
\item (when $C$ is a semiring) $(c_1)(c_2y) = (c_1c_2)y.$
\end{enumerate}

 A \textbf{basis} of an $C$-module $\mcM$ (if it exists) is a set $\{x_i : i \in I\} \subseteq \mcM$ such that any element of $A$ can be written uniquely as a sum $\sum c_i x_i,$ $c_i \in C,$ where almost all $c_i = 0.$
In this case we call $\mcM$ a \textbf{free} $\mcA$-module
of rank $|I|$.

\item   A $\Ccal$-\textbf{bimagma} is  a $\Ccal$-module which is also a bimagma and  satisfies
$$(cy_1)y_2 = c(y_1 y_2) = y_1 (cy_2)$$
for all $c \in \Ccal,$ $y_i \in \mcA.$. A $\Ccal$-\textbf{bimagma ideal} of a $\Ccal$-{bimagma} $\mcA$ is a sub-bimagma which is also an $\mcA$-module.

\item  A \textbf{multiplicative ideal} of a bimagma $\mcA$ is a subset $\mcI \subseteq \mcA$ satisfying $bd,db \in \mcI $, for each $b\in \mcI $ and $d\in \mcA.$

  An \textbf{ideal}  is a sub-semigroup $\mcI \subseteq \mcA$ which is also a multiplicative ideal.

\item  An \textbf{involution} of a $\Ccal$-bimagma $\mcA$ is an anti-automorphism $(*): \mcA\to \mcA$ of order $\le~2,$ i.e.,
$(cb)^* = cb^*$, $(\sum b_i)^* = \sum b_i^*$, $(b^*)^*= b,$ and $(b_1 b_2)^* = b_2^* b_1^*$ for $b, b_i \in \mcA.$
 (We have defined an involution of the first kind.)

\item  A \textbf{semialgebra} over
 $\Ccal$ is a $\Ccal$-bimagma that is
also a semiring.

  \item A map $f:M\to N$ of $\Ccal$-modules is \textbf{module multiplicative} if $f(cy) = c f(y),$ for all $c\in \Ccal,$ $y\in M.$

 \item  Module homomorphisms  are defined as usual. For a
$\Ccal$-module $\mcM$,  the semialgebra of
module homomorphisms $\mcM \to \mcM$ is
denoted as~$\End_{\Ccal} \mcM$. For
notational convenience, we omit the subscript $\Ccal$ when it is understood, and designate $\zero_\mcM \in \End
\mcM$ for the $0$ homomorphism, i.e., $\zero_\mcM(v) = \zero, \ $ for all $v\in \mcM.$ 
\end{enumerate}
\end{definition}

\begin{remark}\label{addingC}
    Any semiring $\mcA$ is a semialgebra over its \textbf{center} $C = Z(\mcA).$
\end{remark}

\subsection{Pairs}\label{subsec:pairs}

 \begin{definition}\label{symsyst} $ $
 \begin{enumerate}\eroman

  \item   A  \textbf{$\Ccal$-pair}  $(\mcA,\mcA_0)$ is a   $\Ccal$-module $\mcA$ with a $\Ccal$-subset $\mcA_0$.
 In particular, if $\Ccal_0 \subset \Ccal$ we have the pair $(\Ccal,\Ccal_0)$,
 which we call the \textbf{base pair}.

  \item   A map $f: (\mcA,\mcA_0)\to  (\mcA',\mcA'_0)$ of pairs is:  \begin{enumerate}\eroman
  \item a \textbf{homomorphism}
if  $f(b_1+b_2)= f(b_1)+f(b_2),$ for all $b_1,b_2\in \mcA .$
\item a  \textbf{weak morphism} if $\sum f(b_i)\in \mcA'_0$
whenever $\sum b_i\in \mcA_0$;  $f$ is
$\mcA_0$-\textbf{injective}   if $\sum f(b_i)\in \mcA'_0$ implies  $\sum b_i\in \mcA_0$.

 \end{enumerate}

 \item   A  \textbf{$(\Ccal,\Ccal_0)$-pair} is a  $\Ccal$-pair $(\mcA,\mcA_0)$  for which $\Ccal_0 \mcA \subseteq \mcA_0$.

 \end{enumerate}
 \end{definition}

\begin{importantnote}
 Any   $\Ccal$-{pair} is automatically a  $(\Ccal,0)$-pair.
 We will identify~$\Ccal$ with $(\Ccal,\zero)$ when appropriate.

 $(\Ccal,\Ccal_0)$ is presumed given, and ``pair'' means $(\Ccal,\Ccal_0)$-pair.
 The justification for this approach is given in
\cite[Note~1.34]{Row16} and \cite{AGR2}.

 Intuitively $\mcA_0$ replaces ``zero.''
 Often $\mcA_0$ is a bimagma ideal of $\mcA.$

 The essential difference with \cite{JMR} and \cite{AGR2} is that here we do not assume that $\Ccal\subseteq \mcA$, and ``tangible elements'' do not play a role here.

 Occasionally we will merely be given a semiring $\mcA$ and an ideal $\mcA_0.$ Then $(\mcA,\mcA_0)$ becomes a pair when we define $C$ as in Remark~\ref{addingC}, and $C_0 = C\cap \mcA_0$.
\end{importantnote}

\begin{definition}\label{basic2}$ $
    \begin{enumerate}\eroman
\item
  A \textbf{bimagma pair}  is a  pair  $(\Acal,\Acal_0)$ for which  $  \mcA$  is a bimagma,
  satisfying $\sum b_i \in \mcA_0$ implies $\sum b b_i \in \mcA_0$.

     \item
  A bimagma pair $(\mcA,\mcA_0)$ is  $\mcA_0$-\textbf{additive} if $\mcA_0$ is an ideal of~$\mcA$.

  \item A \textbf{semiring
 pair} is a bimagma pair, for which $\mcA$ is a semiring.

   \item An $\vep$-\textbf{pair}, for $\vep \in \Ccal,$ is an $\mcA_0$-additive bimagma pair $(\mcA,\mcA_0)$, for which $$xy +\vep yx\in \mcA_0, \ \mbox{for all} \ x,y \in \mcA.$$

\item
 Given a bimagma pair $(\mcA,\mcA_0)$, an \textbf{involution} of $(\mcA,\mcA_0)$ is an involution $(*)$ of $\mcA$ such that  $\mcA_0^* = \mcA_0.$

 \item    A  {pair} $(\mcA,\mcA_0)$ is  \textbf{$\mcA_0$-cancellative} if, for $y\in \mcA,$   $c\in \Ccal,$ $cy \in \mcA_0$ implies $c\in \mcA_0$ or $y\in \mcA_0. $

\item    A  {pair} $(\mcA,\mcA_0)$ satisfies \textbf{$\mcA_0$-elimination} if  $y_0+y_1 \in \mcA_0$ for $y_0\in \mcA_0,$ $y_1\in \mcA,$  implies $y_1\in \mcA_0 $.

   \item A \textbf{homomorphism} of bimagma
pairs $\varphi: (\mcA,\mcA_0)\to (\mcA',\mcA_0')$ is a bimagma homomorphism
 $\varphi: \mcA \to \mcA'$ (i.e., which preserves addition and multiplication), with $\varphi(\mcA_0) \subseteq \mcA_0'$.

\item
  A  (left) \textbf{module pair} $(\mcM,\mcM_0)$ over a pair $(\mcA,\mcA_0)$  is an $\mcA$-module $\mcM$
 together with a subset $\mcM_0$ and a bilinear product $\mcA \times \mcM \to \mcM$ satisfying the following properties for all $b\in \mcA,\ y \in \mcM$:
 \begin{enumerate}
 \item $ \zero y = \zero_\mcM = y\zero,  $
     \item $b\zero_\mcM = \zero_\mcM , $
     \item $b \mcM_0 \subseteq \mcM_0,  $
     \item $ \mcA_0y \subseteq \mcM_0.  $
 \end{enumerate}

     \item
 If $(\mcM, \mcM_0)$ is a  pair, then define $$ \End \mcM _0 = \{f\in \End \mcM: f(\mcM)\subseteq \mcM_0\},$$ and
take $\End (\mcM, \mcM_0)$ to be the pair    $(\End \mcM, \End \mcM_0)$.

  \item A \textbf{sub-pair} of a pair  $(\mcA,\mcA_0)$ is a pair
  $(S,S_0)$ where $S\subseteq\mcA$ and $S_0 \subseteq S \cap \mcA_0.$

    \end{enumerate}
\end{definition}

 \subsection{Substitutes for negation}

 Although we have bypassed negation, we need some  versions to carry out the theory.

\subsubsection{Pre-negation and negation maps} \label{sub:negation}

\begin{definition}[\hspace{1sp}\cite{JMR}]\label{pnm}
A \textbf{pre-negation map} on  a bimagma pair  $(\mcA,\mcA_0)$
is a semigroup endomorphism  $b\mapsto \psi (b)$ of $\mcA$, satisfying the following conditions,  for all $ b,b_1\in \mcA$:
\begin{enumerate}\eroman
\item
$  \psi (bb') = b \psi (b')= \psi (b)b',$
\item $b+ \psi(b) \in \mcA_0 $ for all $b\in \mcA$,
\item $ \psi (\mcA_0) \subseteq \mcA_0. $
\end{enumerate}
\end{definition}

  \begin{definition}[\hspace{1sp}\cite{JMR}]\label{negmap}$ $
 \begin{enumerate}
     \item  A \textbf{negation map} on  a bimagma pair is a pre-negation map, denoted $(-) $, on $(\mcA,\mcA_0)$ of order $\le
2$, i.e., satisfying $(-)((-)b)=b$  for all $b\in \mcA.$

\item  We write $b(-)b'$ for $b +(-)b'.$ 

 \item A $\psi$-\textbf{pair} is a bimagma pair with a pre-negation map $\psi.$
 \end{enumerate}
\end{definition}

\subsubsection{Weak Property N}\label{subsec:N}

We   avoid negation maps, and instead use the following   generalization.

\begin{definition}\label{propN0}$ $
A pair $(\mcA,\mcA_0)$ satisfies
 \textbf{Weak Property N}  if
  for each $b\in \mcA$ there is an element  $b'$ such that $b +b' = b' +b\in \mcA_0$.  

\end{definition}

(We used ``Weak'' here to be consistent with the terminology of \cite{AGR2}. It has nothing to do with ``weak morphisms,'' to be defined below.)
The following very easy example is illustrative.

\begin{example}\label{vep} For  any $\Ccal$-bimagma $\Acal$,
picking any element $\vep$ in $\Ccal$ for which $\one +\vep \in C_0,$ the map $b\mapsto \vep b$ is a pre-negation map $\psi$ of    $(\Acal,(\one +\vep)\Acal)$, which is a $\psi$-pair satisfying Weak Property $N,$ since $b +\vep b = (\one +\vep) b\in \mcA_0.$
\end{example}

 \begin{importantnote} When $\Ccal$ has a negation map, we can take $\vep =(-)\one$ in Example~\ref{vep}.
 In general, the element $\vep$ is a more general version of $(-)\one,$ since we do not require $\vep^2 = \one,$ but nevertheless  $\one +\vep$ replaces $\zero.$
 \end{importantnote}

\subsection{Surpassing relations}

 \begin{definition}\label{ddi}$ $
  A \textbf{surpassing relation} $\preceq$ on a
bimagma  pair  $(\mcA,\mcA_0)$ is a pre-order satisfying the following
conditions for all $b_i,b_i'\in \mathcal A$, $c\in \Ccal$:

  \begin{enumerate}
 \eroman
  \item If $b_1 \preceq b_2$ and $b_1' \preceq b_2'$   then  $b_1 + b_1' \preceq b_2
   + b_2'$ and $b_1  b_1' \preceq b_2
   b_2'.$

      \item If $b_1 + b_0  = b_1'$ for some $b_0\in \mathcal A_0$, then $b_1 \preceq b_1'$.

 \item   If  $c \in \Ccal$ and $b_1 \preceq b_1'$ then $cb_1 \preceq cb_1'.$

      \item When $\mcA$ has a given  negation map $(-)$, if $b_1 \preceq b_1'$ then $(-)b_1\preceq (-)b_1'.$
     \end{enumerate}\end{definition}
 We also write $b_1 \succeq b_2$ to denote that $b_2\preceq b_1.$

  \begin{example}$ $
  \begin{itemize}
      \item
   Our main example in this  paper of a
surpassing relation on a pair $(\mcA,\mcA_0)$, denoted $\preceq_0,$ is given by $b_1\preceq_0 b_2$ iff $b_2 = b_1 +z$ for some $z\in \mcA_0.$
Then
$\mcA_0 =  \{ b\in \mcA: \zero\preceq_0 b \}. $ When $\mathcal{A}_0=0$, the relation $\preceq$ becomes equality.

  \item For $\psi$-pairs, we write $\preceq_\psi$ for $\preceq_0$; i.e.,  $b_1\preceq_0 b_2$ iff $b_2 = b_1 +(b+\psi(b))$ for some $b\in \mcA.$

\item Another example, motivated by hypergroup theory, to be used in Theorem~\ref{Kras}:   Let $ \mathcal P(\mathcal H )$ denote the power set of a set $\mathcal H$. We say that $S_1 \preceq_\subseteq S_2$ if $S_1 \subseteq S_2$.
  \end{itemize}
  \end{example}

 \begin{remark} $ $\begin{enumerate}
\item  Any surpassing relation  $ \preceq$ on a pair $(\mcM,\mcM_0)$  induces a surpassing relation element-wise on  $\End (\mcM, \mcM_0)$, by $f\preceq g$
 if $f(y) \preceq g(y), $ for all $ y \in \mcM.$
 \item  The surpassing relation $\preceq_0$ restricts to a surpassing relation on   sub-pairs. \end{enumerate}
 \end{remark}

\begin{importantnote}\label{superf}
   A surpassing relation $\preceq$ can be useful, since we may generalize
classical formulas by replacing equality by $\preceq$.
 \end{importantnote}

\subsubsection{$\preceq$-morphisms}

\begin{definition}\label{precmor}
A map $f: (\mcA,\mcA_0)\to  (\mcA',\mcA'_0)$ of pairs is a $\preceq$-\textbf{morphism} if satisfies $f(b_1+b_2)\preceq f(b_1)+f(b_2)$ for  $b_1, b_2$ in $\mcA$,  and  $ f(b_1) \preceq  f(b_2)$ whenever  $b_1\preceq b_2$;  $f$ is
\textbf{$\preceq$-injective}  if  $ f(b_1) \preceq  f(b_2)$ implies  $b_1\preceq b_2$.
\end{definition}

\begin{lemma}\label{mor1}
Every $\preceq$-morphism is a weak morphism.
\end{lemma}
\begin{proof}
    If $b_1+b_2 \in \mcA_0$ then $b_1+b_2 \succeq \zero,$ so $f(b_1)+f(b_2) \succeq \zero,$ i.e., $f(b_1)+f(b_2)\in \mcA_0$.
 \end{proof}

 \subsection{Identities and varieties}

 We appeal to some of the notions from universal algebra, without going into the technicalities. Jacobson's book \cite{Jac1980}  is a good resource.

 In brief, an \textbf{$\Omega$-algebra} is a set which has $n$-ary operations which we assume here include addition and multiplication and their bimagma laws, and when appropriate, the negation map.  The 0-ary operations are just  distinguished elements. We also admit identities, i.e., equality of universal atomic formulas  (in terms of the operations). A \textbf{homomorphism} of $\Omega$-algebras is a function which preserves all the given operations.

 We   generalize this notion to an \textbf{$\Omega$-algebra pair} to be a pair $(\mcA,\mcA_0)$ of $\Omega$-algebras such that $\mcA_0$ is a multiplicative ideal of $\mcA$ invariant under  the given unary operations.

 \begin{definition}
 A \textbf{free object} for  $\mathcal{V}$ is some
 $U \in \mathcal{V}$ together with an index set $I$ and a set $X = \{x_i: i \in I\}$,
 such that for any  $\mcA \in \mathcal{V}$ and $\{ b_i: i \in I\}\subseteq \mcA$ there is a unique homomorphism $\Phi:U \to \mcA $ for which $\Phi(x_i) = b_i,$
 for all $ i \in I.$
 \end{definition}

 \begin{example}\label{univex}$ $ Let $I$ be an index set,
 and $I_0\subset I$.
 \begin{itemize}
     \item  The free $\Ccal$-module  of rank $|I|$ was defined in \Dref{basicdefs}, which we denote as~$\Ccal^{(I)}.$

     \item  The free $\Ccal$-pair $(\Ccal^{(I)},\Ccal^{(I_0)})$ of rank $(|I|,|I_0|)$; one can notice that $\Ccal^{(I_0)} = \sum _{i\in I_0}\Ccal x_i$ has a   basis $\{x^i: i \in I_0\}$ which is expanded to a basis $\{x^i: i \in I\}$ of $\Ccal ^{(I)} = \sum _{i\in I}\Ccal x_i$.

     \item  One can take the  free $(\Ccal ,\Ccal _0)$-module  $(\Ccal ,\Ccal _0)^{(I)}= (\Ccal  ^{(I)},\Ccal _0^{(I)} )$ over $(\Ccal ,\Ccal _0)$.

      \item  The free $\Ccal $-module with a formal negation map, of   rank $|2I|$, has a formal basis $$\{x_i: i \in I\}\cup \{y_i: i \in I\} ,$$ where we define $(-)x_i = y_i$ and $(-)y_i = x_i$. (This idea will be pursued in~Example~\ref{Making}.)

     \item The free multiplicative magma $\mathcal{M}(I)$ is constructed as the set of words in the indeterminates~$x_i,\ i \in I$,  without associativity.  Multiplication is juxtaposition, but with putting in parentheses at each stage. To wit, the $x_i$ are words, and if $w_1$ and $w_2$ are words, then $(w_1w_2)$ is a word.  For example, $(x_1(x_2x_3))$ and $((x_1x_2)x_3)$ are different words.

     We get a pair by taking $\mathcal{M}(I)_0$ to be the submagma
     consisting of words, at least one of whose indeterminates is $x_i$ for $i\in I_0.$

     \item The free d-bimagma $\mathcal{F}(I)$  is the magma semialgebra of the free multiplicative magma, i.e., is built from the free module having as basis the free multiplicative magma $\mathcal{M}(I)$, whose multiplication is extended via distributivity, as elaborated below in \Eref{fd}.

      \item The free d-bimagma pair  is  $(\mathcal{F}(I),\mathcal{F}(I)_0)$.

 \item The free semigroup is constructed as the set of words in indeterminates $x_i$, with multiplication being juxtaposition, but without parentheses.

 \item  The free semiring is the semigroup semiring of the free semigroup.

 \end{itemize}
 \end{example}

 Recall that a \textbf{variety} $\mathcal{V}$ in universal algebra is closed under direct products, substructures, and homomorphic images.

 \begin{lemma}\label{var2}$ $\begin{enumerate}\eroman
     \item If $(\mcA, \mcA_0)$ is a pair and $\varphi:\mcA \to \overline{\mcA}$ is a homomorphism, then  $(\overline{\mcA},\ \overline{\mcA}_0 \! :=  \!  \varphi(\mcA_0))$  is a pair. Furthermore, a surpassing relation $\preceq$ on $(\mcA, \mcA_0)$ induces a surpassing relation $(\overline{\mcA}, \overline{\mcA}_0)$, by putting $\bar b _1\preceq \bar b_2$
     if $b_1\preceq b_2.$
     \item If $(\mcA_i, \mcA_{i0})$ are   pairs for each $i\in I,$ then the direct product $(\prod_{i\in I} \mcA_i,\prod_{i\in I} \mcA_{i0}) $ is a pair, and surpassing relations on each pair $(\mcA_i, \mcA_{i0})$ induce a surpassing relation on  $(\prod_{i\in I} \mcA_i,\prod_{i\in I} \mcA_{i0}) $, componentwise.
     \item Generalizing (ii), for any filter $\mathcal{F}$ on $I$, one can define the \textbf{reduced product} (cf.~\cite{CK}) $$ \prod_{i\in I} \mcA_i / \mathcal{F}$$ by saying $(b_i)\cong (c_i)$ if $\{ i\in I: b_i = c_i\} \in \mathcal{F},$ and
     then get the \textbf{reduced~pair}  $ (\prod_{i\in I} \mcA_i / \mathcal{F},   \prod_{i\in I} \mcA_{i0} / \mathcal{F}),$ which inherits the  surpassing relation;
  i.e., $(b_i)\preceq (c_i)$ if and only if $\{i:  b_i \preceq  c_i \}\subseteq \mathcal{F}.$
     \item Negation maps are preserved under reduced products.
      \item Weak Property N is preserved under reduced products.
 \end{enumerate}
 \end{lemma}
 \begin{proof}
 The verifications are routine, noting that a filter is closed under finite intersections.
 \end{proof}

For the $\preceq$-theory,  we introduce the relation $\preceq$ into the language, even though it introduces difficulties.

\begin{remark}
    \begin{enumerate}

        \item   In general  a surpassing relation $\preceq$  need not  pass to homomorphic images, because one could conceivably have $b_1 \preceq b_2$ and $b_3 \preceq b_4$ with $\bar b_2 = \bar b_3,$ but $\bar b_1 \not\preceq \bar b_4.$

\item The surpassing relation
 $\preceq_0$ does remain a surpassing relation under homomorphic images.
 Namely, if $b_2 = b_1 + c_0$ and
 $b_4 = b_3 + d_0$ for $c_0, d_0 \in \mcA_0,$ with $\bar b_2 = \bar b_3,$ then
 $$\bar b_4 = \bar b_3+ \bar d_0 = \bar b_2+ \bar d_0 =  \bar b_1 + \bar c_0 + \bar d_0,$$
 i.e.,  $\bar b_1  \preceq_0 \bar b_4.$

\item The surpassing relation $\preceq_0$  need not  pass to
    sub-pairs, because we might lose null elements.

    \end{enumerate}
\end{remark}

 \begin{definition}\label{ide}$ $
 \begin{enumerate}\eroman
     \item
  An \textbf{identity}  is a universal atomic formula $f(x_1,\dots, x_m)=g(x_1,\dots, x_{m'}).$
      \item
  A $\preceq$-\textbf{identity} is a universal atomic formula  $f(x_1,\dots, x_m) \preceq g(x_1,\dots, x_{m'}).$

 \end{enumerate}
 \end{definition}

 \begin{proposition}\label{precvar}
 Any class of $\Omega$-algebras defined by identities and $\preceq_0$-identities on pairs has free objects.
 \end{proposition}
 \begin{proof}
 For identities we simply impose the relations on the elements of $U$ and $U_0$ by means of congruences, as is customary in universal algebra. For $\preceq_0$-identities  $f \preceq_0 g$ we adjoin fresh distinct indeterminates $y_{f,g}$ to $U_0$ and impose the relations $f + y_{f,g} = g.$
\end{proof}

\begin{example}\label{fd}
The elements of the free d-bimagma are obtained by repeated addition and multiplication. Distributivity permits us to rewrite any element
$f(x_1,\dots, x_m)$ as a sum $\sum_j h_j$ of \textbf{ monomials}, i.e., products of the $x_i$, together with some coefficient.  We say that the monomial $h_j$ is \textbf{multilinear of degree $m$} if each~$x_i$, $1 \le i \le m$, appears exactly once in $h_j;$ moreover $f$ is \textbf{multilinear of degree $m$} if each of its monomials $h_j$ is {multilinear of degree $m$}.

 In each case, an identity or $\preceq$-identity is \textbf{multilinear of degree $m$}
     if in the notation of \Dref{ide}, $m={m'}$ and both $f$ and $g$ are {multilinear of degree $m$}.
\end{example}

\begin{lemma}\label{ml} To verify a multilinear identity or $\preceq$-identity in an $\mcA_0$-additive bimagma pair, it is enough to
 check it on a spanning set $S$ over $\Ccal $.
   \end{lemma}
   \begin{proof}
Just write each element $b_i$ as $\sum_j c_{ij}s_j$ for $s_j\in S,$ and open up the expression.
   \end{proof}

\section{The  basic theory of  Lie pairs}\label{LieD0}

We introduce adjoints as a preparation for the Lie theory.

\begin{definition}\label{ad1}
For a bimagma pair $(\mcA,\mcA_0)$, and $x\in \mcA,$
      we define
the \textbf{adjoint maps} $\ad_x : \mcA \to \mcA$ and $\ad^\dagger_x
: \mcA \to \mcA $, by  $\ad_x(y) = xy$ and $\ad_x^\dagger(y)
= yx$ for $y\in \mcA.$
 We also define $\Ad_\mcA = \{ \ad_x: x\in \mcA\}$,  $\Ad^\dag_\mcA = \{ \ad^\dag_x: x\in \mcA\}$,  $\AD(\mcA) = \Ad_\mcA +\Ad^\dag_\mcA,$ and
  ${\AD(\mcA) }_0 = \{ f \in \AD(\mcA) : f(\mcA)\subseteq \mcA_0\}.$
\end{definition}

\begin{remark}$ $
\begin{enumerate}\eroman
    \item  Suppose $(\mcA,\mcA_0)$ is a  pair $(\mcA,\mcA_0)$, and $x\in \mcA.$ Then
$\ad_x, \ad^\dag_x \in \End (\mcA, \mcA_0)$.
    \item
$({\AD(\mcA) } ,{\AD(\mcA) }_0 )$ is a sub-pair of $(\End \mcA, \End \mcA_0)$.
\end{enumerate}
\end{remark}

 \subsection{Lie brackets and Lie pairs}

We are ready to bring in the Lie
bracket, the focus of this paper.

 \begin{definition}\label{LieD}$ $\begin{enumerate}\eroman
    \item
A $\mcL_0$-\textbf{Lie bracket} on a pair $(\mcL,\mcL_0)$ is a map, written  $[\phantom{xy}]: \mcL
\times \mcL \to \mcL$,
satisfying the following \textbf{Lie bracket axioms}, for all $x,y\in \mcL$,
with $\ad_x$ and $\ad_x^\dagger $ as in \Dref{ad1}:

 \begin{enumerate}\eroman
     \item  $\ad_x(x) \in \mcL_0,$ i.e., $[xx]\in\mcL_0$,

   \item $\ad_x + \ad_x^\dagger \in {\End (\mcL, \mcL_0)}_0,$ i.e., $[xy]+[yx]\in\mcL_0$  (the intuition being that  right multiplication acts like the negation of  left multiplication),

  \item    $\ad_{[xy]} +\ad_x \ad_y ^\dag  + \ad_y ^\dag \ad_x\in \End (\mcL, \mcL_0)_0, $
      called the \textbf{Jacobi $\mcL_0$- identity}.

    \item   [(c$'$)] $\ad^\dag_{[xy]} +\ad_y ^\dag \ad_x  + \ad_y \ad_x^\dag  \in \End (\mcL, \mcL_0)_0, $
      called the \textbf{reflected Jacobi $\mcL_0$- identity}.

          \item $\ad_{cx} = c\ad_{x}$ for all $c\in C;$
      \item If $\sum_i  x_i \in \mcL_0,$ then   $\sum_i \ad^\dag_y (x_i)\in \mcL_0$, and $\sum_i \ad_y (x_i)\in \mcL_0$ for all $y\in \mcL$.
 \end{enumerate}

  \end{enumerate}
   \end{definition}

   \begin{remark}\label{spel0}$ $
  \begin{enumerate}\eroman

     \item   For $z\in \mcL,$ Axiom (c) of (i) translates to
   \begin{equation}\label{spel}
       [[xy]z] +[x[yz]]+ [y[zx]] \in \mcL_0.
   \end{equation}

   \item  Interchanging $y$ and $z$ in \eqref{spel2}
yields $\ad_x \ad_y =  \ad^\dag_y \ad^\dag_x .$ If this holds then axiom $(c')$ is superfluous.
 \end{enumerate}
\end{remark}

\begin{lemma}
    Axiom (b) is implied by (a)  in any Lie pair satisfying  $\mcL_0$-elimination.
    \end{lemma}

    \begin{proof}
    $[xx]+[yy] +[xy]+[yx] =[(x+y)(x+y)] \in \mcL_0.$ But  $[xx]+[yy] \in \mcL_0$ by (a), so $[xy]+[yx]\in \mcL_0.$
    \end{proof}

    \begin{definition} A \textbf{\weakLie\ pair} is a pair $(\mcL,\mcL_0)$ endowed
with a  $\mcL_0$-Lie bracket. A \textbf{Lie pair} is a \weakLie\  pair $(\mcL,\mcL_0)$ whose  $\mcL_0$-Lie bracket is $C$-bilinear, also satisfying the condition that if $\sum_i  x_i \in \mcL_0,$ then $\sum_i \ad_{x_i} \in {\End (\mcL, \mcL_0)}_0$, $\sum_i \ad^\dag_{x_i} \in {\End (\mcL, \mcL_0)}_0$. The   $\mcL_0$-Lie bracket is \textbf{$\dag$-reversible} if $\ad_x ^\dag \ad_y= \ad_x \ad^\dag_y,$ for all $x,y\in \mcL.$
    \end{definition}

The $\dag$-reversibility translates to
    \begin{equation}\label{spel2}
   [[yz]x]= [x[zy]].
   \end{equation}

Here are  other desirable properties that hold  in the classical situation.

\begin{definition}\label{LieD8}$ $
\begin{enumerate}\eroman
    \item
 The  \textbf{reflected} $\mcL_0$-Lie bracket is defined as  $[xy]^\dag = [yx].$ $(\mcL,\mcL_0)^\dag$ is $(\mcL,\mcL_0)$ with the reflected $\mcL_0$-Lie bracket.

      \item The   $\mcL_0$-Lie bracket is \textbf{$\mcL_0$-reversible} if  $[xy]\in \mcL_0$ implies  $[yx]\in \mcL_0$,
 \item The   $\mcL_0$-Lie bracket is   \textbf{$\mcL_0$-symmetric} if $\ad_x  = \ad^\dag _x $ for $x\in \mcL_0.$

  \item  The  \textbf{reflected \weakLie\ pair} is the \weakLie\ pair
  $(\mcL,\mcL_0)$ with reflected $\mcL_0$-Lie bracket.

           \item A \textbf{\Lie pair} is a Lie pair   $(\mcL,\mcL_0)$ whose
$\mcL_0$-Lie bracket is  $\dag$-reversible.
\end{enumerate}
\end{definition}

\begin{remark}
      The reflection of a Lie pair is a Lie pair.
\end{remark}
\begin{importantnote}   $ $\begin{enumerate}\eroman
     \item We always assume that $\mcL \ne \mcL_0,$ since otherwise the axioms are vacuous.
      \item If $\mcL_{0} = \{\zero\},$ the axioms revert to classical Lie theory.
  \item We often view $\mcL$ as a bimagma, whose multiplication is the Lie bracket. Usually $\mcL_0$  is a sub-bimagma.
    \item
 We need  bilinearity to determine a Lie bracket in terms of   products of basis elements. But
there is an interesting example (Theorem~\ref{Kras}) arising from hyperrings, which fails bilinearity yet satisfies  part of distributivity.
    \item
In general neither $\dagger$-reversibility nor $\Lcal_0$-reversibility holds, cf.~the cross
 product examples of
 \S\ref{cp}.
But if $\Lcal_0$-elimination holds, then $\Lcal_0$ reversibility holds. Indeed, if $[xy]\in \Lcal_0$, then $[xy]+[yx]\in\Lcal_0$,  implying $[yx]\in\Lcal_0$.
\end{enumerate}
\end{importantnote}

 \subsubsection{Lie brackets  on a free module over a base pair $(\Ccal ,\Ccal _0)$}\label{sec311}

The following observation provides a method of constructing  Lie brackets on a free module over a base pair $(\Ccal ,\Ccal _0)$, especially in the finite dimensional case.

\begin{lemma}\label{fgen}
If $\mcL$ is a free module over a base pair $(\Ccal ,\Ccal _0)$ with basis $\{b_i : i\in I\}$, then the Lie bracket can be defined in terms of the products
$$
[b_ib_j]=\sum_k c^k_{ij}b_k, \quad c^k_{ij}\in \Ccal ,
$$
and $(\Lcal,\Lcal_0)$ is a Lie pair if and only if
 these coefficients satisfy the following axioms for each ${i,j,k,m}\in I$:
\begin{enumerate}
    \item  $ c_{ii}^m\in \Ccal _0,$
   \item  $  c_{i,j}^m+ c_{j,i}^m \in \Ccal _0$,
   \item $\sum _{l} (c_{ij}^lc_{lk}^m+ c_{kj}^lc_{li}^m+c_{ki}^lc_{lj}^m)\in \Ccal _0,$
\item $\sum _{l} (c_{ij}^lc_{lk}^m+  c_{kj}^lc_{li}^m+c_{ki}^lc_{lj}^m)\in \Ccal _0$.
\end{enumerate}
Moreover $\dag$-reversibility holds if and only if
$\sum _{l} c_{ij}^lc_{lk}^m= \sum _{l} c_{kj}^l c_{li}^m,$ for all ${i,j,k,m}.$\end{lemma}
\begin{proof}
The only axiom in \Dref{LieD}(i) which is not multilinear is (a),
which says for all $c_i\in C$ that $$ \left(\sum c_i b_i\right)^2= \sum _i c_i^2 b_i^2 +\sum c_{i}c_j(b_ib_j+b_jb_i)  \in \mcL_0,$$ implied by $b_i^2 \in \mcL_0$  and (b).

So we need $\sum _m c_{ii}^m b_m \in \Ccal _0,$ which means  $c_{ii}^m \in \Ccal _0$ for each $m\in I.$

We check all the other axioms on basis elements $b_i,b_j,b_k.$ Axiom (b) requires that $ \sum_{ m} (c_{i,j}^m+ c_{j,i}^m)b_m = b_ib_j+b_j b_i \in \Ccal _0,$
 so $c_{i,j}^m+ c_{j,i}^m \in \Ccal _0$ for each $m\in I.$

The Jacobi $\mcL_0$-identity reads as
$$
[[b_ib_j]b_k]+[[b_jb_k]b_i]+[[b_k,b_i]b_j]=
\sum _{l,m} (c_{ij}^l c_{lk}^m+   c_{jk}^lc_{li}^m+   c_{ki}^lc_{lj}^m)b_m,$$

so we need $\sum  _{l}( c_{ij}^lc_{lk}^m+  c_{jk}^lc_{li}^m+   c_{ki}^lc_{lj}^m)\in \Ccal _0$,
for each $i,j,k,m\in I.$

Similarly, $\dag$-reversibility means
$\sum _{l} (c_{ij}^lc_{lk}^m+  c_{kj}^lc_{li}^m+c_{ki}^lc_{lj}^m)\in \Ccal _0$ for each $i,j,k,m\in I.$
\end{proof}

In this way, any Lie pair is determined via the $\Lcal$ valued matrix
$$
([b_ib_j])_{1\leq i,j\leq n}
$$
such that all the diagonal elements $[b_ib_i]\in \Lcal_0$ and $[b_ib_j]+[b_jb_i]\in \Lcal_0$ ($\mcL_0$-skew symmetry).

\begin{definition}$ $
  \begin{enumerate}\eroman
\item  For
$V,W$ subsets of a bimagma $\mcL$, we define $[VW]$ to be the $\Ccal $-subspace of $\mcL$
generated by $\{ [vw]: v\in V, \ w\in W\}.$


\item An ideal of a Lie pair (with respect to the Lie bracket)
$(\mcL,\mcL_0)$~is also called a \textbf{Lie ideal}, for emphasis.

\end{enumerate}
\end{definition}
\begin{lemma} $ $\begin{enumerate}\eroman
\item Any bimagma sub-pair $(W,W_0)$ of a Lie pair $(\mcL,\mcL_0)$ is itself a Lie pair, which is $W_0$-reversible (resp.~$\dag$-reversible) if $(\mcL,\mcL_0)$ is  $\mcL_0$-reversible (resp.~$\dag$-reversible).

    \item $([W\mcL],[W\mcL]\cap \mcL_0)$ is a Lie ideal of $(\mcL,\mcL_0)$ for any ideal $W$ of $\mcL$.
 \item   Let $\mcL':= [\mcL\mcL]$.
 $(\mcL',\mcL\cap \mcL_0)$ is a Lie ideal of $(\mcL,\mcL_0)$ that satisfies Weak Property N.
\end{enumerate}
\end{lemma}
   \begin{proof}
  (i) We show that $(W,W_0)$ is a Lie pair, by verifying the conditions of \Dref{LieD}(i), for $x,y,z \in W$.
   \begin{enumerate}
       \item   $[xx] \in \mcL_0 \cap W = W_0.$
        \item $[xy]+[yx] \in \mcL_0 \cap W = W_0.$
        \item Likewise the Jacobi $W_0$-identity (and its reflection) and $\ad_x =\ad_x^\dag$ for $x\in W_0$ are a fortiori, as well as $\dag$-reversibility.
        \item The other axioms are clear.
   \end{enumerate}

   (ii)  Clearly $[x[yw]],[[yw]x] \in [W\Lcal]$ for $w\in W.$

 (iii) Using (ii), we only need to verify the Weak Property N, which is clear since
 $[xy]+[yx]\in \mcL_0.$
 \end{proof}

\subsubsection{Negated Lie pairs}

 \begin{definition}
      A \textbf{negated Lie pair} is a Lie pair with a negation map
      (cf. Definition~\ref{negmap}), such that $[yx]=(-)[xy]$ for all $x,y$.
 \end{definition}

 \begin{remark}
   Negated  Lie pairs are rather restrictive. For example any negated  Lie pair is reflexive and  $\mcL_0$-symmetric.
 \end{remark}

\subsection{Lie pairs with a surpassing relation}\label{LieD70}

One could introduce a surpassing relation $\preceq$.

\begin{definition}\label{LieD7}$ $
\begin{enumerate}
    \item A $\preceq$-\textbf{Lie bracket} with regard to a surpassing relation $\preceq$ is a    $\mcL_0$-Lie bracket   also satisfying, for all $x,x_i,y\in \mcL$:

\begin{enumerate}
  \item (The
 Jacobi $\preceq$-identity) $\ad_{[xy]} \preceq \ad_x \ad_y+ \ad_y \ad^\dag _x$.
 \item $ \ad_{\sum _i x_i } \preceq  \sum _i \ad_{x_i}  ,$
 \item $\ad _x(\sum y_i)\preceq \sum \ad _x(y_i)$ and  $\ad ^\dag_x(\sum y_i)\preceq \sum \ad^\dag _x(y_i)$ for all $x,y_i \in \mcL$.
 \item $ \ad^\dag_{\sum _i x_i } \preceq  \sum _i \ad^\dag_{x_i}$.
 \item If $x \preceq y$, then $ \ad_x \preceq  \ad_ y$ and $\ad^\dag_x \preceq \ad^\dag_y$.
\end{enumerate}

  \item  A  $\preceq$-\textbf{weak Lie pair} is a  weak  Lie pair  $(\mcL,\mcL_0)$   with  a
surpassing relation~$\preceq$.

    \item  A  $\preceq$-\textbf{Lie pair} is a  Lie pair  $(\mcL,\mcL_0)$   with  a
surpassing relation~$\preceq$.
\end{enumerate}
\end{definition}

\begin{importantnote}
The classical Lie theory has equality holding in   (1), but
we find this too restrictive to obtain a workable algebraic theory for semialgebras.
 \end{importantnote}

\begin{remark}\label{spel3} Assume that $(\mcL,\mcL_0)$ is a reversible Lie pair.
One can rewrite Axiom~(1)(a) as  \begin{equation}
    \label{J1} [[xy]z]\preceq [x[yz]]+[y[zx]].
\end{equation}
\end{remark}
\begin{lemma}\label{LLieD} $ $
Assume that $(\mcL,\mcL_0)$ is $\mcL_0$-reversible (\Dref{LieD8}).
\begin{enumerate}\eroman

 \item     The Jacobi $\preceq$-identity also is  equivalent to each of:
\begin{enumerate}
 \item[(d$'$)]   $\ad _y \ad^\dag_x  \preceq  \ad _{[xy]}+ \ad _x \ad^\dag_y  $.

\item[(d$''$)]   $\ad^\dag_{[xy]} \preceq \ad^\dag_x \ad^\dag_y+ \ad^\dag_y \ad _x$.

\end{enumerate}
\item
If   $[[xy]z] + w = [x[yz]]+[y[zx]]$ for $w \in \mcL_0,$ then $$[z[yx]] + w = [[zy]x]+[[xz]y].$$ \end{enumerate}
\end{lemma}
\begin{proof} Use Equation \eqref{J1} throughout.

(i) To obtain Axiom (d$'$) switch $y$ and $z.$ The reverse argument gives us \eqref{J1}  from (d$'$).

 To obtain Axiom (d$''$), plug \Dref{LieD8}(ii)  into each term of \eqref{J1}, and then exchange $x$ and~$y$. The reverse argument gives us \eqref{J1} from~(d$''$).

(ii)
When $[[xy]z] + w = [x[yz]]+[y[zx]]$ for $w \in \mcL_0,$ $$[z[yx]] + w =[[xy]z] +w = [x[yz]]+[y[zx]] = [[zy]x]+[[xz]y].$$
\end{proof}

\begin{remark}
Axioms (d$'$) and (d$''$) remain consistent when $\mcA_0$-symmetry holds.
\end{remark}
\subsection{Categories involving Lie  pairs}

There are three natural kinds  of morphisms of Lie  pairs,
each of which defines its category.

\begin{definition}\label{morp} A $\Ccal $-module homomorphism $f: (\mathcal L,\mcL_0)\to (\mathcal N,
\mcN_0)$ of Lie pairs  is a \textbf{Lie bracket map} if $f(\mcL_0)\subseteq \mcN_0$ and $f([b_1  b_2 ]) = [f(b_1)f(  b_2)],$ for all $b_1,b_2\in \mcL$. A Lie bracket map $f$ is a \textbf{weak Lie morphism}, $\preceq$-\textbf{Lie morphism}, resp.~\textbf{Lie homomorphism}, if $f$ is a weak  morphism, resp.~$\preceq$-morphism, resp.~homomorphism.
    \end{definition}

  \begin{lemma}$ $\begin{enumerate}\eroman
     \item The Lie   pairs and their weak Lie   morphisms
comprise a category.

   \item The $\preceq$-Lie    pairs  and their Lie    $\preceq $-morphisms
comprise a subcategory of (i).

  \item
The Lie    pairs and their Lie   homomorphisms
comprise a subcategory of (ii).
  \end{enumerate}
\end{lemma}
  \begin{proof}
  One checks easily that the composition of two Lie homomorphisms is a Lie homomorphism, and likewise for weak Lie  morphisms and  Lie  $\preceq$-morphisms. The other assertions are by Lemma~\ref{mor1}.
  \end{proof}

 \begin{importantnote}
      Although Lie homomorphisms are the definition from  universal algebra, the first category, using weak Lie morphisms, fits best into the general theory of  pairs.
 \end{importantnote}

 \begin{example}\label{pro}
 Suppose $(\mcL,\mcL_0)$ is any Lie pair, and $\mcL_0\subset \mcL_1 \subset \mcL$. Then $(\mcL,\mcL_1)$
 also is a  Lie pair, and the identity map can be viewed as a Lie homomorphism   from $(\mcL,\mcL_0)$ to~$(\mcL,\mcL_1)$.
 Likewise for weak Lie pairs and $\preceq$-Lie pairs.
 \end{example}

\section{Lie pair constructions}\label{subsec:MLC}

In this section we show how  celebrated examples of Lie theory can be generalized to Lie pairs.

\subsection{$\psi$-Lie pairs from associative and pre-Lie $\vep$-pairs}\label{vepp}

In the classical theory of Lie algebras one knows that for each associative algebra (and more generally for pre-Lie algebras \cite{Bai}), the additive commutator makes it into a Lie algebra. In our situation, this cannot work since we do not have negation.
 Nevertheless, there is an analogous procedure for pairs.

 \begin{definition} $ $\begin{enumerate}\eroman
     \item A $\psi$-\textbf{Lie pair} is a  Lie pair having a pre-negation map $\psi$.

 \item A \textbf{strong} $\psi$-\textbf{Lie pair} is a
    $\psi$-{Lie pair}  satisfying $[yx] = \psi([xy])$ for each $x,y\in \mcL.$

\item For $\vep\in C$, an $\vep$-\textbf{Lie pair} is a  Lie pair  satisfying   $x+\vep x \in \mcL_0$ for all $x \in \mcL$ and $[xy]+[yx]\in (1+\vep)\Lcal\subseteq \Lcal_0$.

      \item  For $\vep\in C$, a \textbf{strong}  $\vep$-\textbf{Lie pair} is a $\vep$-\textbf{Lie pair}   satisfying $[yx] = \vep([xy])$ for each $x,y\in \mcL$.
 \end{enumerate}
 \end{definition}

When $\mcL_0 = (\one +\vep)\mcL$ for $\vep \in C$, (iii), (iv)  are special cases of (i),(ii) respectively, taking $\psi$ to be the map $x \mapsto \vep x$.

 \begin{lemma}
      Any strong $\psi$-Lie pair $(\mcL,\mcL_0)$ is $\dag$-reversible, and is $\mcL_0$-reversible.
 \end{lemma}
 \begin{proof}
        $$[[xy]z] = \psi(\psi[[xy]z] )= [[\psi(x)y]\psi(z)] = [z[yx]], $$
  so   $\dag$-reversibility holds.

  If $[xy] \in \mcL_0,$ then $[yx] = \psi([xy])\in \mcL_0.$\qed
 \end{proof}
 \begin{theorem}\label{pLie} Given a semiring $\psi$-pair $(\Rcal,\Rcal_0)$,
  define $(\Rcal,\Rcal_0)_\psi := (\Rcal,\Rcal_0),$  endowed with the Lie bracket   defined by
$[xy]_\psi  =xy+\psi( y)x$. Then $(\Rcal,\Rcal_0)_\psi $ is a $\psi$-Lie pair.
Moreover, $(\Rcal,\Rcal_0)_\psi $ is a strong $\psi$-Lie pair when   $\psi^2 = \one_\Rcal.$
 \end{theorem}
\begin{proof}
We verify the axioms in \Dref{LieD}(i).

\begin{enumerate}
    \item[(a)] $[xx]=x^2+\psi( x)x= x^2 +\psi( x^2)\in \Rcal_0.$

    \item[(b)] $[xy]+[yx]=xy+\psi( y)x+yx+\psi( x)y=xy+yx +\psi(xy+yx)\in \Rcal_0.$

    \item[(c)]
        $\begin{aligned}[t]
            [[xy]z]{ }&{ }+[[yz]x]+[[zx]y] \\
            &= (xy+\psi( y)x)z+\psi( z)(xy+\psi( y)x)+(yz+\psi( z)y)x \\
            &{\quad}+ \psi( x)(yz +\psi( z)y) + (zx+\psi( y)x)y+\psi( y)(zx+\psi( x)z)\\
            &= \psi( xy  z)+  yxz+\psi( zxy)+\psi^2 (zyx)+ yzx+\psi( zyx)\\
            &{\quad}+\psi( xyz)+\psi^2( xzy)+ zxy+\psi( xzy)+\psi( yzx)+\psi^2( yxz)\\
            &=\psi \big((xyz+yzx+zyx)+\psi(zyx+xzy+yxz)\big)\\
            &=\psi \big([xy]z+[yz]x+[zx]y +\psi([xy]z+[yz]x+[zx]y)\big)\in \Rcal_0.
        \end{aligned}$

    \item[(c$'$)] is analogous, and (d),(e) are easy.
\end{enumerate}

When $\psi^2 = \one_\Rcal$, we have $[yx] = yx+\psi( x)y  =   \psi(xy+\psi( y)x) = \psi([xy])$.
\end{proof}

\begin{remark}$ $
    \begin{enumerate}\eroman
        \item For instance, for a $\Ccal$-semialgebra $\Rcal$,  pick any element $\varepsilon \in \Ccal$, and define  $\Ccal_0 = \Ccal(\one +\vep) $  and
  $\Rcal_0 = (\one +\vep)\Rcal $. Then $(\Rcal,\Rcal_0)$ is a semiring pair satisfying the hypothesis of the theorem, taking the pre-negation $\psi$ to be $b \mapsto \vep b.$

  \item As an example for when  $\mcA_0$-reversibility   holds, in
any bimagma with a negation map $ (-) $, we write $[b,b']$ for the
\textbf{Lie commutator} $bb' (-) b'b.$
    \end{enumerate}
\end{remark}
  \begin{corollary}\label{Pois1}
      Any semiring
   pair $(\mcA,\mcA_0)$ with a negation map $(-)$ becomes a negated Lie
 pair, denoted by $(\mcA,\mcA_0)^{(-)}$,
  under the Lie product $[bb'] = [b,b']$, which also satisfies  $\dag$-reversibility and $\mcA_0$-reversibility.
  \end{corollary}
  \begin{proof}
  Take $\vep = (-)\one$.
  \end{proof}

We also can obtain a $\preceq$-version, by  extending the \textbf{Leibniz $\preceq$-identities} given in
\cite[Lemma~2.35]{Row16} to a pre-negation map $\psi$ cf.~\Dref{pnm}:

\begin{remark}
    If $\psi$ is a  pre-negation map on a semiring, then $$\psi (x_1 \dots x_n) = x_1 \dots x_{i-1} \psi(x_{i})x_{i+1}\dots x_n$$ for all $i$, by induction on $n.$
\end{remark}
\begin{lemma}[\textbf{Leibniz $\psi$-identities}]\label{Leib1}  In any semiring $\mcA,$ defining $[x,y]_\psi  = xy+\psi( yx)$.
\begin{enumerate}\eroman
    \item $
[x,y]_\psi  z +y[x,z]_\psi =[x,yz]_\psi  +yxz +\psi(yxz)  .$
\item $
z[x,y]_\psi  + [x,z]_\psi  y =[x,zy]_\psi  +zxy +\psi(zxy)  .$

\item
 $ [x,[yz]_\psi ]_\psi  +yxz+zxy+  \psi(yxz+zxy) = [ [x,y]_\psi ,
z]_\psi  +[y,[x,z]_\psi ]_\psi   .$
In particular, $ [x,[y,z]_\psi ]_\psi \preceq_\psi [ [x,y]_\psi ,
z]_\psi +[y,[x,z]_\psi ]_\psi   .$
\end{enumerate}
 \end{lemma}
 \begin{proof}
 (i) As in \cite[Lemma~2.35]{Row16}, we compute:
 $$[x,y]_\psi  z +y[x,z]_\psi  = (xy+\psi( yx))z + y(xz + \psi( z)x) = xyz +\psi( yzx) + yxz +\psi(yxz) .$$

 (ii) By symmetry.

 (iii) Add (i) to (ii).

 \end{proof}

\subsubsection{Pre-Lie $\psi$-pairs}

The construction of Lie algebras from pre-Lie algebras  also can be extended to Lie pairs. Recall that $\psi$ is a pre-negation map.

\begin{definition}\label{ass} The $\psi$-\textbf{associator} in a $C$-bimagma is given by $$(x,y,z)_\psi  :=  (xy)z + \psi( x(yz)).$$
An $\mcA_0$-additive bimagma pair $(\mcA,\mcA_0)$ is a \textbf{pre-Lie $\psi$-pair} if $(x,y,z)_\psi +\psi((x,z,y)_\psi)  \in \mcA_0$ for all $x,y,z \in \mcA.$
\end{definition}

\begin{theorem}\label{preLir}
Any $\mcA_0$-additive pre-Lie $\psi$-pair $(\mcA,\mcA_0)$ becomes an  $\psi$-Lie pair under the Lie bracket $[xy]_\psi  :=xy+\psi( yx).$

If $\psi$ is invertible, the reverse bracket of $[xy]_\psi $ is $\psi([xy]_{\psi^{-1}}).$
\end{theorem}
  \begin{proof}
  We modify the proof of Theorem~\ref{pLie}.
  (a) and (b) are the same, and to get (c) we have

   \begin{equation}
  \begin{aligned}
 [[xy]z]+[[yz]x]+[[zx]y] &= (xy)z+\psi( yx)z+\psi( z)(xy)+\psi( z)\psi(yx)+(yz)x\\
 &+ \psi( (zy))x + \psi( x)(yz) +\psi( x)\psi(zy) + (zx)y\\
 &+\psi(xz)y+\psi( y)(zx)+\psi( y)\psi(xz)\\
 &=(x,y , z)_ \psi +  ( y,z,x)_\psi  +(z,x,y)_\psi  +\psi((z,y,x)_\psi\\
 &+(x,z,y)_\psi  + (y,x,z)_\psi  )\\
 &=  ((x,y , z)_ \psi +  \psi( (x,z , y)_ \psi) + ( ( y,z,x)+  ( y,x,z) _ \psi )\\
 &+((z,x,y)_\psi  + \psi( (z,y,x)_\psi)  )
    \end{aligned}
\end{equation}
which is in $ \Rcal_0.$

$(c')$ is analogous, and (d),(e) are easy.

Finally, $[xy]_\psi ^\dag = [yx]_\psi  = yx +\psi( xy) = \psi( (xy  +\psi^{-1}( y)x).$
  \end{proof}

  \subsection{Lie pairs from semiring pairs with involution}$ $\label{sec:invo}

We can also get examples from involutions.
Let $\Rcal$ be an associative $\Ccal$-semialgebra equipped with an involution $*:\Rcal\to \Rcal$, cf.~\Dref{basic2}.
Define $\Lcal$ to be  $\Rcal$, endowed with the bracket:
\begin{equation}
\label{br}
[xy]=xy+y^*x
\end{equation}
and $\Lcal_0$ be the   $\Ccal$-module
$$
\sum _{x\in \Rcal}   \Rcal(x+x^*) + \sum _{x\in \Rcal}   (x+x^*)\Rcal + \sum _{x\in \Rcal}\Rcal   (x+x^*) \Rcal .
$$
In particular $x(y+y^*)$,  $(x+x^*)y$, and $x(y+y^*)z$ belong to $\Lcal_0$ for any arbitrary choice of $x,y,z \in\Rcal$.
\begin{theorem}\label{Linv}
  $(\Lcal,\Lcal_0)$ as defined above is a   Lie  pair.
\end{theorem}
\begin{proof}
Let us check the axioms \ref{LieD}(i).

(a) $[xx]=x^2+x^*x=(x+x^*)x\in \Lcal_0.$

(b)  $[xy]+[yx] = xy+y^*x+yx+x^*y=(x+x^*)y+(y+y^*)x \in \Lcal_0.
$

(c) The Jacobi $\mcL_0$- identity holds, since
\begin{equation}\label{eq:jcbst}
\begin{aligned}
\relax [[xy]z] +[[yz]x]+[[zx]y]&=(xy+ y^*x)z+z^*(xy+y^*x)+(yz+z^*y)x+x^*(yz+z^*y)\\
&+(zx+x^*z)y+y^*zx+x^*z)\\
&=\underbrace{(x+x^*)}yz+\underbrace{(y^*+y)}zx+\underbrace{(z^*+z)}xy\\
&+y^*\underbrace{(x+x^*)}z+z^*\underbrace{(y+y^*)}x+x^*\underbrace{(z+z^*)}y.
\end{aligned}
\end{equation}
Since each bracketed term belongs to $\Lcal_0$, the left side of \eqref{eq:jcbst}  belongs to $\Lcal_0$.

(d) Follows from the fact that $\Lcal_0$ is an ideal.
\end{proof}

\begin{importantnote}
Theorem~\ref{Linv} is the $\mcL_0$-version of skew symmetric elements
(defined by $x+x^* =\zero)$, since here we have stipulated $x+x^* \in \mcL_0.$
\end{importantnote}

\subsubsection{$\vep$-Skew symmetric pairs\footnote{We could do this for a general pre-negation map $\vep$ if we stipulate that $\psi$ preserves the involution, i.e., $\psi(x^*) = \psi(x)^*$.}}

More generally, let us now fix $\vep\in C$ and define the bracket on $\Rcal$ by
\begin{equation}
\label{br1}
[xy]=xy+\vep y^*x.
\end{equation}
Take $\mcL = R$ and stipulate that $\Lcal_0$ contains $\{x+\vep x^* : x\in \Rcal\}$.
\begin{theorem}\label{Linv1}
$(\Lcal,\Lcal_0)$ is a  Lie pair under the bracket of \eqref{br1}.
\end{theorem} \begin{proof}
Again let us check  the axioms of Definition \ref{LieD}(i).
First of all
\[
[xx]=x^2+\vep x^*x=x(x+\vep x^*) \in \Lcal_0.
\]
This proves that axiom \ref{LieD}(i)(a) holds.

To check (b),
$[xy]+[yx]=xy+\vep y^*x+yx+\vep x^*y=(x+\vep x^*)y+(y+\vep y^*)x$ which belongs to $\Lcal_0$ because $(x+\vep x^*)$ and $(y+\vep y^*)$ do.

For (c), the Jacobi $\mcL_0$-identity, we use again expression \eqref{eq:jcbst} with $\vep$ inserted in the appropriate places.

(d) and (e) are obvious.
\end{proof}

\subsection{The ``classical'' Lie pairs}$ $

We can now describe the paired version of the classical Lie algebras $A_n, B_n, C_n, D_n$.
We need the {\it trace}  $\tr{(A)} := \sum a_{ii}$ of a matrix $A = (a_{ij}).$

\begin{lemma}
 $\tr(AB) = \tr(BA)$ for matrices $A,B$ over $\Ccal$.
\end{lemma}
\begin{proof}
$\tr(AB)= \sum a_{ij}b_{ji} = \sum b_{ji}a_{ij} = \tr(BA), $
since $\Ccal$ is commutative.
\end{proof}

\begin{theorem}\label{clas1} Fix $\vep \in \Ccal$ such that $\vep+1 \in \Ccal_0.$ (See footnote 2.)

\begin{enumerate}\eroman
    \item The paired version of the classical Lie algebra $A_n$ is given by the \textbf{special linear pair}   ${\rm sl}_n :=(\Lcal,\Lcal_0),$
    where $\Lcal = \{ x\in M_n(\Ccal) : \tr(x)\in \Ccal_0\},$ and $\Lcal_0$ is obtained as in Theorem~\ref{preLir}.
    \item The  $\vep$-paired version of the classical Lie algebra $B_n$ is given by $\mathfrak{so}_{2n+1}^{(\vep)} :=(\Lcal,\Lcal_0),$
    where  $\Lcal =  x\in  M_{2n+1}(\Ccal)  : x +\vep x^T \in M_{2n+1}(\Ccal_0)\},$ and $\Lcal_0$ is obtained as in Theorem~\ref{Linv}.
    \item The $\vep$-paired version of the classical Lie algebra $C_n$ is given by  $\mathfrak{sp}_{2n}^{(\vep)} :=(\Lcal,\Lcal_0),$
    with  $\Lcal =  \{ x\in M_{2n}(\Ccal) : Jx +\vep x^TJ \in M_{2n+1}(\Ccal_0)\},$ where $J$ is the matrix $   \bpmtx
     0 & 1 \\ \vep & 0\epmtx,$ and $\Lcal_0$ is obtained as in Theorem~\ref{Linv}.
    \item The  $\vep$-paired version of the classical Lie algebra $D_n$ is given by $\mathfrak{so}_{2n}^{(\vep)} :=(\Lcal,\Lcal_0),$
    where  $\Lcal =  \{ x\in M_{2n}(\Ccal) : Jx +\vep x^TJ \in M_{2n}(\Ccal_0)\},$ and $\Lcal_0$ is obtained as in Theorem~\ref{Linv}.
\end{enumerate}
\end{theorem}
\begin{proof}
(i) By Theorem~\ref{preLir}, noting that ${\rm sl}_n^{(\vep)}$ is closed under $[\phantom{w}]_\psi $ since $\tr{AB+\vep BA} \in \Ccal_0.$

(ii), (iv) By Theorem~\ref{Linv}.

(iii) Also by Theorem~\ref{Linv}, using the involution $x\mapsto J^{-1}x^T J$ (formally adjoining $\vep^{-1}$ if necessary).
\end{proof}

The exceptional Lie pairs could also be defined, but this is effected most easily via the Jordan version.

\subsection{Non-classical examples}\label{subsec:ME}

Other Lie pairs  cannot be obtained by means of Lie commutators.

\begin{example} The $\Ccal_0$-skew $3\times 3$ matrices   deserve  further analysis. In $\Lcal:=M_3(\Ccal)$ we  consider matrices of type $J_0,J_1,J_2$, which  depend on two $\Ccal_0$-constrained parameters of $\Ccal$, namely
$$
J_0:=\left\{\left.J_0(a,a'):=\bpmtx 0&a&0\cr a'&0&0\cr 0&0&0\epmtx\,\right|\, a+a'\in\Ccal_0\right\}\in \Ccal^{3\times 3},
$$
\medskip
$$
J_1:=\left\{\left.J_1(b,b')=\bpmtx 0&0&b\cr 0&0&0\cr b'&0&0\epmtx\,\right|\, b+b'\in\Ccal_0\right\}\in \Ccal^{3\times 3},
$$
\medskip
$$
J_2:=\left\{\left.J_2(c,c')=\bpmtx 0&0&0\cr 0&0&c\cr 0&c'&0\epmtx\,\right|\, c+c'\in\Ccal_0\right\}\in \Ccal^{3\times 3}.
$$

Clearly each element of $\Lcal$ can be written (not uniquely) as a  $\Ccal$-linear combination of a matrix of type $J_0$, one of type $J_1$, and one of type $J_2$.
We claim that
\begin{enumerate}
\item[(a)] $[J_iJ_i]\in \Lcal_0$;
\item[(b)] $[J_i,J_j]\subseteq J_{(i+j)\mod 3}$.
\end{enumerate}
We know already that $(\Lcal,\Lcal_0)$ is a Lie pair but nevertheless it is instructive
 to perform explicit computations to obtain  Property (b).
\begin{eqnarray*}
[J_0(a,a'),J_1(b,b')]&=&J_0(a,a')J_1(b,b')+J_1(b,b')^TJ_0(a,a')\cr\cr
&=&\bpmtx0&a&0\cr a'&0&0\cr0&0&0\epmtx\bpmtx0&0&b\cr 0&0&0\cr b'&0&0\epmtx+\bpmtx0&0&b'\cr 0&0&0\cr b&0&0\epmtx\bpmtx0&a&0\cr a'&0&0\cr0&0&0\epmtx\cr\cr\cr
&=&\bpmtx0&0&0\cr 0&0&ab\cr0&ab'&0\epmtx=aJ_2(b,b')\in J_2.
\end{eqnarray*}

Clearly the same argument holds for the other possible choices of indices, and b) is proven.

\medskip
\noindent
\end{example}
To go further in producing examples, we start with the most fundamental ones.
Given a pair $(\mcL, \mcL_0) $ we will write its product as the bracket $[xy]$,
with the hope of showing that it is a Lie bracket. Write
$\mcL':= [\mcL\mcL]$.
\begin{definition}
     A   pair  $(\mcL, \mcL_0) $ is $\mcL_0$-\textbf{Lie abelian} if $\mcL'\subseteq \mcL_0.$ More generally, $(\mcL, \mcL_0) $  is $\mcL_0$-\textbf{Lie nilpotent} of \textbf{index} $2$ if $[\mcL \mcL']\subseteq \mcL_0$.
\end{definition}

 \begin{lemma}\label{nil1}
 ${}$
\begin{enumerate}\eroman
    \item  All  $\mcL_0$-Lie nilpotent pairs of {index} $2$ satisfy the Jacobi $\mcL_0$- identity.
    \item If  $(\mcL, \mcL_0) $ is $\mcL_0$-{Lie abelian}, then  the Jacobi $\preceq$-identity holds.
\end{enumerate}
 \end{lemma}
 \begin{proof}
 (i) All the terms are in $\mcL_0$.

 (ii)    $[xy] , [yz]\in \mcL_0$ imply
  $[[xy]z] = [[zy]x]]= [x[yz]],$ and $[y[zx]]\in \mcL_0,$
  which in turn implies that $[[xy]z] \preceq [x[yz]+[y[zx]].$
 \end{proof}

 \subsubsection{ Low dimensional examples}

 If
 $\Lcal=\bigoplus_{i=1}^n\Ccal x_i$ we say  that $ \Lcal $ has dimension $n$.
    Many of the lowest dimensional examples lack negation maps. Blachar  \cite[\S 2.3]{Bl} provided the 3-dimensional examples of Lie pairs over a semifield $\Ccal $ having a negation map, so we shall only consider Lie pairs   over a  semifield pair $(\Ccal ,\Ccal _0)$.  \begin{example}\label{exab}
 The only 1-dimensional example is supplied by the trivial algebra $\mcL = \Ccal x,$ with $[xx]=0.$
 \end{example}
 \begin{example}\label{ex:ex55}
The   2-dimensional examples where
 the Lie pair is  $\mcL_0$-{Lie abelian}, so the Jacobi $\mcL_0$-identity holds. Let $\mcL = \Ccal x + \Ccal y.$
 \begin{enumerate}
  \item $\mcL_0 = \{\zero\};$ then we get the classical examples in \cite{Jac}.
     \item $\mcL_0 = \Ccal y$, where $[xx] = y$ and one of the following holds:
     \begin{itemize}
      \item   $ [xy] = [yx] = [yy] =y$.
   \item   $ [xy] = [yx] = y$ and $[yy] = 0.$
                \item $[xy] = [yx] = [yy] = 0.$
     \end{itemize}

             \item Now $\mcL_0 = \Ccal ( \mu x + \nu y)  $, with $\mu,\nu \ne 0.$ One example is $\mu+\nu =1,$ $[xx]=[yy]=0,$ and $[xy] =  [yx] =  \mu x + \nu y.$

             $[x[xy]] = \nu [xy]$ and $[y[xy]] = \mu[xy].$

             $[x[yx]] = \mu [xy]$ and $[y[yx]] =[[xy]x]= \nu[yx].$
              \end{enumerate}
              \end{example}

         A relevant $3$-dimensional example, whose aim is to recover the situation of the cross product,   will be studied separately in Section \ref{cp} below.

    \begin{example}
        Some $4$-dimensional examples. Let $\mcL = \Ccal x \oplus \Ccal y\oplus \Ccal z_1 \oplus \Ccal z_2,$ and  let $\mcL_0 =\Ccal z_1 \oplus \Ccal z_2$, where $ [xx] = [yy] = 0$,  $ [xy] = z_1,$ $ [yx] = z_2,$ and \begin{enumerate}
          \item $[xz_i]= [yz_i] = [z_iz_j] = 0$ for all $i,j$.
           \item (The  Heisenberg pair) $[xz_1]= [z_1y]= [z_2x]=[yz_2]= z_1$, $[z_iz_j] = 0$ and additionally $[xz_2]= [z_2y]= [z_1x]= [yz_1]= z_2.$
            \item $[xz_1]= [z_1y]= [z_2x]=[yz_2]= [z_1z_2]=z_1$, $[xz_2]= [z_2y]= [z_1x]= [yz_1]= [z_2z_1]=  z_2.$
      \end{enumerate}
      Example \ref{ex:ex55} (1) satisfies the Jacobi $\mcL_0$-identity, by computation. The other Lie pairs are $\mcL_0$-{Lie abelian}, so satisfy the Jacobi $\mcL_0$-identity.
 \end{example}

 \subsubsection{Filiform pairs}\label{fili}

Another large class of examples is provided by the {\em filiform   algebras} \cite{BaGoRem}, an important class of nilpotent Lie algebras    which
 has a {\em Verne basis} $\{x_1,\dots, x_n\}$ satisfying
\begin{enumerate}\eroman \item $[x_1x_i]=  x_{i+1},$ $1\le i \le n-1$,
        \item $[x_1, x_{n}] = 0,$
    \item $[x_i x_j] = \sum_{k\ge i+j} c_{i,j} x_k,$ $c_{i,j}\in C.$
\end{enumerate}

\begin{definition}
  A \textbf{filiform pair} is a Lie pair $(\Lcal,\Lcal_0)$ such that $\Lcal$ is a free $(C,C_0)$-module with basis
  $x_1,\dots,x_n$ satisfying the conditions
  \begin{enumerate}
    \item
  $[x_1x_i]=  x_{i+1},$ $1\le i \le n-1$,
        \item $[x_1, x_{n}] \in \Lcal_0,$
    \item $[x_i x_j] \in   \sum_{k\ge i+j} c_{i,j} x_k +\mcL_0 ,$ where $ c_{i,j} + c_{j,i} \in \mcL_0.$
\end{enumerate}
\end{definition}
(In particular $ [xx], [xy]+[yx] \in \mcL_0$, for all $x,y\in {\mcL}$.)

\begin{example} Let us see a few instances.
 \begin{enumerate}
\item The \textbf{standard $3$-dimensional filiform pair}   has generators $x_1,x_2,x_3$.    Choose  $\ell_{21}$ as $c_1x_1+c_2x_2+c_3x_3\in \Lcal$ arbitrarily ($c_i\in \Ccal $), and take $\Lcal_0=\Ccal (x_3+\ell_{21})$, together with  the relations
$$
[x_1x_2]=x_3,\qquad [x_2x_1]=\ell_{21},\qquad  [x_ix_j]=0 \quad \text{otherwise}.
$$
The Jacobi $\mcL_0$- identity  then  holds trivially,  checked on generators.
\item More generally choose $\ell_{13},\ell_{31},\ell_{23},\ell_{32},$ and define $\Lcal_0$ as being the $\Ccal $-submodule generated by $(\ell_{ij}+\ell_{ji}, \ell_{21}+x_3)$ ($i+j\geq 4$).
Then the commutators
$$
[x_1x_2]=x_3,\qquad [x_2x_1]=\ell_{21}, \quad[x_1x_3]=\ell_{13}, $$ $$  \quad [x_3x_1]=\ell_{31}, \quad [x_2x_3]=\ell_{23},\quad [x_3x_2]=\ell_{32},
$$
define a Lie pair $(\Lcal,\Lcal_0)$.
\end{enumerate}
 \end{example}

\subsubsection{The cross
 product Lie pair}\label{cp}

In any reasonable theory of Lie pairs one should be able to recover  the classical example of the Lie algebra $(\RR^3,\times)$, the cross product in the three-dimensional real vector space. We will do it via the procedure described in Section~\ref{sec311}.


\begin{example}[The cross product]\label{ex:lie220}
 Let
$$
\mcL=\Ccal b_0\oplus \Ccal b_1\oplus \Ccal b_2.
$$

Take arbitrarily two arbitrary  $3$-tuples  $
(c_0,c_1,c_2)$ and $(d_0,d_1,d_2)$ in $\Lcal^3,$ not necessarily $\Ccal$-linearly independent. 
We define a Lie bracket on $\Lcal$, depending on the choice of $(d_i)$ and $(c_i)$ (i.e. a $6$- parameter family) by encoding it into a
$\Lcal$-valued $3\times 3$ matrix $A:\Lcal\times\Lcal\lra \Lcal$ given by:
\be
A:=\begin{pmatrix}d_0&b_2&c_1\cr
c_2&d_1&b_0\cr
b_1&c_0&d_2
\end{pmatrix}\in \Lcal^{3\times 3}\cong \Lcal^*\otimes \Lcal^*\otimes\Lcal,\label{eq:matC}
\ee
stipulating that
$$
[b_i b_j]=A(i,j), \qquad  0\le i,j\le 2.
$$
We    obtain a Lie pair  generically, imitating the natural structure of the cross product. For this reason we define $$\Lcal_0(A)=\Ccal\langle d_i, b_i+c_i, b_ic_i\rangle.
$$



The notation reflects the fact that the submodule $\Lcal_0(A)$ of $\Lcal$ depends on the matrix~$A$.
Let us check that  $(\Lcal, \Lcal_0(A))$   satisfies the axioms of  $(a)$, $(b)$  $(c)$, $(c')$, $(d)$ and $(e)$ of Definition~\ref{LieD}(i).
To this purpose, we first compute the product of two generic elements
$$
x=x_0b_0+x_1b_1+x_2b_2\qquad  and\qquad  y=y_0b_0+y_1b_1+y_2b_2
$$
of $\Lcal$, using the multiplication matrix.
A simple computation yields:
\begin{equation}
    \begin{aligned}
 \relax    [x y]& = x_0y_0d_0+x_1y_1d_1+x_2y_2d_2 \\ &= x_1y_2b_0+x_2y_1c_0+x_2y_0b_1+x_0y_2c_1+x_2y_0b_1+x_0y_1b_2+x_1y_0c_2
    \end{aligned}
\end{equation}
\begin{enumerate}
  \item[(a)] Let us check that $[x x]\in \Lcal_0(A)$.
  Indeed,
  $$
  [x x]=x_0^2d_0+x_1^2d_1+x_2^2d_2+x_0x_1(b_2+c_2)+x_0x_2(b_1+c_1)+x_1x_2(b_0+c_0)\in\Lcal_0(A);
  $$

  \item[(b)] Let us check that $ad_x(y)+ ad^\dagger_x(y)=[x y]+[yx] \in\Lcal_0(A)$.
\begin{eqnarray}
  [x y]+[y x]=& 2x_0y_0d_0+2x_1y_1d_1+2x_2y_2d_2+(x_0y_1+x_1y_0)(b_2+c_2)\cr
  &+(x_0y_2+x_2y_0)(b_1+c_1)+(x_1y_2+x_2y_1)(b_0+c_0)\in\Lcal_0(A)
 \end{eqnarray}

\item[(c)] We   now come to   the Jacobi identity. Besides the generic elements $x$ and $y$ mentioned before, let $z=z_0b_0+z_1b_1+z_2b_2$ and consider the Jacobi sum
\be
(xy)z+(yz)x+(zx)y.\label{eq4:Jidcr}
\ee
Expanding \eqref{eq4:Jidcr} in terms of the component of $x$, $y$ and $z$, one easily get
\begin{eqnarray*}
&&[[x y] z]+[[y  z]  x]+[[zx]    y]\cr\cr&=&\sum_{0\leq i,j,k\leq 2}x_iy_jz_k\big([[b_i  b_j]b_k]+[[b_j  b_k] b_i]+[[b_k  b_i]   b_j]\big).
\end{eqnarray*}
 For each choice of $(i,j,k)\in\{0,1,2\}^3$ we have basically two cases.
\begin{enumerate}
\item[i)] $(i,j,k)$ is either an even or odd permutation of $(0,1,2)$.
In the even case we have
\begin{eqnarray*}
&&[[b_i  b_j] b_k]+[[b_j  b_k] b_i]+[[b_k  b_i] b_j]\cr &=&[b_k  b_k]+[b_i  b_i] +[b_j  b_j]=d_0+d_1+d_2\in\Lcal_0(A).
\end{eqnarray*}
In the odd case:
$$
[[b_i  b_j]  b_k]+[[b_j  b_k]  b_i]+[[b_k  b_i]  b_j=c_kb_k+c_ib_i+c_jb_j\in \Lcal_0(A).
$$
\item[ii)] If $i=j$ then
$$
[[b_i  b_i]  b_k]+[[b_i  b_k]  b_i]+[[b_k  b_i] b_i]=[d_i  b_k]+[(b_j+c_j) b_i]\in \Lcal_0(A).
$$
\item[iii)] The case $i=k$ works the same as in   (ii).
\end{enumerate}
\end{enumerate}


 The $\dag$-reversibility does not hold in general. We compute $[[b_0b_1]b_2]= [b_2b_2] = d_2,$   whereas $[[b_2b_1]b_0] = [c_0b_0],$ so in general we need $[c_ib_i]=d_i,$ which  normally fails.


\end{example}
\begin{example}\label{cr2}
Generalizing Example~\ref{ex:lie220}, let $V$ be a free module over  $\Ccal$ and also let $A:V\otimes V\to V$ be a $V$-valued bilinear form over  $\Ccal$. If $V=\bigoplus_{1\leq i\leq n} \Ccal b_i$, let us denote   $A =(a_{ij})$ for $a_{ij}\in V$, where $b_ib_j=a_{ij}$, and let
$$
\Lcal=\bigoplus_{i<j}\Ccal\cdot a_{ij}
$$

Define
\be
\Lcal_0(A):=\Ccal\cdot \langle a_{ij}+a_{ji},\, a_{ii},\, a_{ij}a_{ji}\rangle\label{eq:LCgen}
\ee
i.e., $\Lcal_0(A)$ is the $\Ccal$-submodule
spanned by the expressions listed in \eqref{eq:LCgen}. Then $(\Lcal, \Lcal_0(A))
$ is a Lie pair. The verification works the same as in the case of $n=3$ (Example \ref{ex:lie220}), so we omit it.
\end{example}

\begin{remark}
If $\Acal$ is an algebra (i.e., with additive inverses) and $c_i=-b_i$ and $d_i=0$, the matrix $C$ as in \eqref{eq:matC} defines the usual cross product (for  $\Acal=\RR$).
\end{remark}

 \subsection{Krasner type}\label{Kras0}

One can also insert some hypergroup theory into the theory of Lie pairs.

\begin{theorem}\label{Kras}$ $
\begin{enumerate}\eroman
    \item  (Inspired by \cite{krasner}) Let $R$ be a semiring, and $G$ a normal multiplicative subgroup of~$R$. Pick $\vep \in R.$ Then $H = R/G$ is a hyper-semiring, and let  $\mcA = \mathcal{P}(R/G)$, i.e., the elements $S \in \mcA $ are unions
    $\cup a_iG$ of cosets of $R$. In other words, if $a\in S$ then $a_ig\in S$ for each $g\in G.$ Addition is defined by $$\boxplus a_iG = \{\sum a_ig_i: g_i \in G\}.$$
     $(\mcA,\mcA_0)^{(\vep)}$ of Theorem~\ref{pLie}, and $\mcA_0 = \{ S \in \mathcal{P}(R/G): \zero \in S\}.$
    $H$
    satisfies all of the axioms of a reversible weak $\preceq_{\subseteq}$-Lie pair,  under the Lie bracket $[aG\, bG] = [a,b]G$.

        \item In (i), we could  take a Lie multiplicative ideal $M$ of $R$ and  instead take  \[
        \mcA_0 = \{ S \in \mathcal{P}(R/G):  S \cap M \ne \emptyset\}.
        \]
    \item Let $R$ be a semiring with an involution $(*)$, and $G$ a normal multiplicative symmetric subgroup of~$R$.  Then the analog of $H$ of (i), in Theorem~\ref{Linv},  is a weak Lie pair,  with
surpassing relation~$\subseteq$.
         \item In (iii), we could  take a symmetric Lie multiplicative ideal $M$ of $R$ and  instead take  $\mcA_0 = \{ S \in \mathcal{P}(R/G):  S \cap M \ne \emptyset\}.$
\end{enumerate}
\end{theorem}

\begin{proof}  (i) We get the Lie product in $R$ as in  {\cite[Proposition~10.7]{Row16}}. \cite[Proposition~5.18]{AGR2} yields associativity of addition.

(ii) Analogous to (i).

(iii) $\mcA/G$ is a hyper-semiring, as in \cite{AGR2}. Then we apply (i),
defining
\[
[S_1,S_2] = \{[a_{i1},a_{j2}]: a_{i1}\in S_1,\ a_{j2}\in S_2\},
\]
and have a weak  Lie pair  with
surpassing relation~$\subseteq$.

(iv) Analogous to (iii).
\end{proof}
\begin{remark}
$(\boxplus_i a_iG)(\boxplus_j b_jG) \subseteq \boxplus_{i,j} a_ib_jG$   in each of the Krasner-type constructions, in view of \eqref{Leib1}, which shows that there formally are more terms in the right side of Definition~\ref{ddi}(ii) than the left, and the extra ones are paired off.
\end{remark}

\section{Doubling}\label{dou}
\begin{example}\label{Making}(Abstract doubling of a $C$-module, also see \cite[Example~1.7(iii)]{AGR2}). This is a way to create a   pair   with a negation map, from any $C$-module $\mcA.$
\begin{enumerate}
    \item Define $\widehat{\mathcal A} := \mathcal A
\times \mathcal A$ with  pointwise addition. We think of the first component as a positive copy of $\mcA,$ and the second component as a negative copy of $\mcA.$

 \item Define multiplication in $\widehat{\mcA} $ by the \textbf{twist action}
\begin{equation}\label{mult1}
    (b_1, b_2)(b'_1, b'_2) = (b_1b'_1+b_2b'_2,\, b_1b'_2+b_2b'_1).
\end{equation}


\item $\widehat { \mathcal A }$   has the ``switch'' negation map given
 by $(-)(b_1,b_2)  = (b_2,b_1)$.

  \item If $ \mcA$ is a $C$-module, then $\widehat { \mathcal A }$ is a $\widehat{C}$-module with the respect to the twist action
  $$(c_1,c_2)(b_1,b_2) = c_1b_2 +c_2b_2, c_1b_2+c_2b_1).$$
  Note that $(-)(\one,\zero) = (\zero,\one).$

\item If $f,g: (\mcA,\tT)\to (\mcA',\tT')$ are homomorphisms, then  define $(f,g):  \widehat{\mathcal A}\to   \widehat{\mathcal A}'$ given by $(f,g)(b_1,b_2) = (f(b_1)+g(b_2),f(b_2)+g(b_1)).$
\end{enumerate}
  \end{example}

 \begin{lemma}\label{supdd}
  Any doubled d-bimagma $\widehat{\mcA}$ is $\Z_2$-graded as $(\mcA \times \{ \zero \}) \, \oplus (\{ \zero \} \times \mcA) $.
  \end{lemma}
\begin{proof}
$\mcA \times \{ \zero \}$ is the ``$+$'' part, and $ \{ \zero \} \times \mcA $ is the ``$-$'' part. We need d-bimagmas  to decompose multiplication according to the grading.
\end{proof}

\begin{remark}
As in \cite{Row16} one could obtain a pair by
    defining
   $$\widehat{\mcA}_0 =  \Diag := \{(b,b): b \in \mcA\},$$
   noting that $(b_1,b_2)(-)(b_1,b_2) =(b_1,b_2)+(b_2,b_1) = (b_1+b_2,b_1+b_2).$
\end{remark}

\subsection{Doubling
 a   pair}

As in \cite{AGR2},   we rather modify the doubling construction when working in the category of pairs, as follows:

\begin{example}[Doubling
 a   pair]\label{Making8}$ $\begin{enumerate}\eroman
     \item  Given a pair  $(\mcA,\mcA_0)$,
   we  obtain a pair $\widehat{(\mcA,\mcA_0)} := (\widehat{\mcA}, \widehat{\mcA}_0)$ by
    defining
   $$\widehat{\mcA}_0 =  \Diag + \{(b_1,b_2): b_1+b_2 \in \mcA_0\}.$$
     \item If $(\mcA,\mcA_0) $ is a $C$-pair, then $(\widehat{\mcA}, \widehat{\mcA}_0)$ is a $\widehat{C}$-pair under the twist action.
 \end{enumerate}
 \end{example}

 \begin{lemma}\label{idea} If $\mcA_0$ is an ideal of $\mcA,$ then
  $\widehat{\mcA_0}$ as defined in \Eref{Making8} is an ideal of  $\widehat{\mcA}$.
 \end{lemma}
 \begin{proof}
 By \eqref{mult1}, noting that if $b_1+b_2 \in \mcA_0$ then $$b_1b'_1+b_2b'_2+b_1b'_2+b_2b'_1 = (b_1+b_2)(b'_1+b'_2)\in \mcA_0.$$
 \end{proof}

\begin{proposition}\label{idgr}
Any multilinear identity or $\preceq$-identity of an $\mcA_0$-additive bimagma pair  (See Definition \ref{symsyst}) $(\mcA, \mcA_0)$ also holds in the doubled pair.
\end{proposition}
 \begin{proof}
By Lemma~\ref{ml} we need only check homogeneous elements, and  they are preserved via the grading.
 \end{proof}

  \subsection{Negated Lie pairs from a semiring}

Motivated by Theorem~\ref{pLie},
we construct a  Lie pair from any semiring, with the switch a  negation map.

\begin{example}\label{Making1}(Lie bracket on a doubled pair)\label{bimd}

Building on \Eref{Making8},
we can define the  Lie bracket $$[(x_1, y_1)(x_2, y_2)] = (x_1x_2 + y_1 y_2 +x_2y_1 + y_2x_1,\ x_1 y_2   +y_1x_2  +x_2x_1  +y_2  y_1  ).$$
 \end{example}

 \begin{theorem}\label{Liedb}
 If $\mcA$ is a semiring then the Lie bracket of \Eref{bimd} makes $(\widehat{\mcA},\widehat{\mcA}_0)$ a    $\preceq_0$-Lie pair, where $\widehat{\mcA}_0 = \{(a,a): a \in \mcA\}$.

 If $(\mcA,\mcA_0)$ is a semiring pair over $(C,C_0)$, then $(\widehat{\mcA,\mcA_0})$ of \Eref{Making8} is a \Lie pair.
 \end{theorem}
\begin{proof}  We get the axioms of   \Dref{LieD}(i) by passing to $\hat{A} $ and applying Theorem~\ref{Pois1}.

We could also verify them directly:
$$[(x,y),(x,y)] = ( xx  +  y y  + xy  +  yx ,  xy  + y  x  + xx  +  yy )\in \hat{\mcA}_0 .$$

\begin{equation}\begin{aligned} \relax  [(x_1, y_1),& (x_2, y_2)]+     [(x_2, y_2),(x_1, y_1)] \\ & = (x_1x_2 + y_1 y_2 +x_2y_1 + y_2x_1,\ x_1 y_2   +y_1x_2  +x_2x_1  +y_2  y_1  )\\ & \quad +(x_2x_1 + y_2y_1  + x_1y_2+y_1x_2 ,\ x_2y_1  +y_2x_1    +x_1x_2  + y_1 y_2  )\in \hat{\mcA}_0,\end{aligned}
 \end{equation}
and
$$[ [(x_1, y_1), (x_2, y_2)] (x_3, y_3)]=[ [(x_3, y_3), (x_2, y_2)] (x_1, y_1)]$$ holds by symmetry of the definition.

 To prove the Jacobi $\preceq_0$-identity, we need to show that
 $$[(x_1, y_1), (x_2, y_2)], (x_3, y_3)]\preceq_0 [ (x_1, y_1)[ (x_2, y_2) (x_3, y_3)]]+ [[ (x_2, y_2), (x_3, y_3)(x_1, y_1)]],$$
which is straightforward but lengthy. \end{proof}

$\widehat{\mcA}_0$-reversibility may fail since $[(x_1,y_1)(x_2,x_2)] = (x_1x_2+y_1x_2, x_1x_2+y_1x_2)$ whereas $[(x_2,x_2)(x_1,y_1)] = (x_2x_1+x_2y_1, x_2x_1+x_2y_1).$

\subsection{Doubling  of Lie pairs}

We use the method of doubling to construct a negation map for a Lie pair.

  \begin{theorem}\label{dLie}
 If  $(\mcL,\mcL_0)$ is a  Lie pair, then $(\widehat{\mcL},\widehat{\mcL}_0)$ is a     Lie pair,  and there is a Lie homomorphism $(\mcL,\mcL_0)\to (\widehat{\mcL},\widehat{\mcL}_0)$ given by $y \mapsto (y,0).$

 If $[\phantom{xy}]$ is a $\preceq$-Lie bracket on $(\mcL,\mcL_0)$, then $[\phantom{xy}]$ naturally induces a $\preceq$-Lie bracket on  $(\widehat{\mcL},\widehat{\mcL}_0)$.

  If  $(\mcL,\mcL_0)$ satisfies $\dag$-reversibility (resp.~$\mcL_0$-reversibility), then so does $(\widehat{\mcL},\widehat{\mcL}_0)$.
 \end{theorem}

\begin{proof}  We verify the axioms of   \Dref{LieD}(i).
$$[(x,y)(x,y)] = ([xx] + [y y] ,[ xy] +[y  x] )\in \mcA_0 \times \mcA_0 \subseteq \hat{\mcA}_0.$$

 $$[(x_1,x_2)(y_1,y_2)] +[(y_1,y_2)(x_1,x_2)] = $$ $$ ([x_1y_1]+[x_2y_2],[x_1y_2]+[x_2y_1]) +([y_1x_1]+[y_2x_2],[y_2x_1]+[y_1x_2]) \in  \hat{\mcA}_0.$$

If $y_1+y_2 \in \mcA_0$ then $[(x_1,x_2)(y_1,y_2)] \in \hat{\mcA_0}$ since $$[x_1y_1 ] + [x_2 y_2]+ [x_1 y_2]   +[x_2y_1]  = [(x_1+x_2)(y_1+y_2)]\in {\mcA}_0. $$

 The other defining identities (as in \Rref{spel0}) and $\preceq$-identities (as in \Rref{spel3}) are  multilinear and thus pass to $(\widehat{\mcL},\widehat{\mcL}_0)$ by \Lref{idgr}.
 \end{proof}

 \begin{importantnote}
 The doubled Lie pair need not be a negated  Lie pair even though it has a  negation map. Indeed, $$(-)[(x_1,x_2)(y_1,y_2)]= (-)([x_1y_1]+[x_2y_2],[x_1y_2]+[x_2y_1]) = ([x_1y_2]+[x_2y_1],[x_1y_1]+[x_2y_2])$$ whereas
 $$[(y_1,y_2)(x_1,x_2)]=  ([y_1x_1]+[y_2x_2],[y_2x_1]+[y_1x_2]).$$
 \end{importantnote}

\section{Representing Lie pairs inside semiring pairs}\label{PBW}\label{sec:PBWx}

As always, we work with pairs over $(C,C_0).$ Our goal in this section is
to embed a Lie pair in an appropriate associative pair. In order to obtain such a pair, we need to consider tensor pairs.

\subsection{The tensor d-bimagma and free Lie pairs}\label{sec:tsflp}

We would like to add free Lie pairs to the list of Example~\ref{univex}. We need a preliminary.
Tensor products of modules over semialgebras can be defined in the usual universal way, cf.~\cite{Ka1}.  But we do not require this generality.

\begin{example}\label{tens}$ $
\begin{enumerate}
    \item
When constructing the tensor d-bimagma $T(V)$, rather than having it   associative, we take the free d-bimagma
given by tensor multiplication of monomials, cf.~\Eref{univex}. In other words,    let $V^{\otimes m}$
denote all tensor powers of~$V$ over $C$, distinguished by parentheses, in the sense that $(V \otimes V)\otimes V$ and $V \otimes (V \otimes V)$ are distinct; for example, $$V^{\otimes 3}: = (V \otimes V)\otimes V \ \oplus \ V \otimes (V \otimes V).$$

To emphasize nonassociativity, we put parentheses around each monomial. We set
$T(V): =  \bigoplus_{m\ge 1} V^{\otimes m}$, with multiplication  defined by juxtaposition, i.e., define
$((h_1)(h_2)) = (h_1) \otimes (h_2),$  for monomials $(h_1)$ and $(h_2)$. For example if
$(h_1),(h_2) \in V^{\otimes 2}$ then writing $(h_i) = (v_i \otimes w_i)$ we get $$(h_1)(h_2) = (v_1 \otimes w_1)\otimes (v_2 \otimes w_2).$$

Thus $V^{\otimes m}$ is spanned over tensor products of the $x_i;$ these
are customarily called \textbf{pure simple tensors}. A \textbf{simple tensor} is a  {pure simple tensor} with a coefficient from $C$.

We   form a d-bimagma pair $(T(V),T(V)_0 )$ over a pair $(C,C _0 )$  by putting $T(V)_0 $ to be the subspace of $T(V)$ spanned by:

\begin{enumerate}
    \item[(a)] all simple tensors containing a factor in $V_0$, and
    \item[(b)] all   simple tensors with coefficients from $C_0,$
\end{enumerate}
\noindent clearly an ideal of  $T(V).$ Note that (b)  is $0$ when $C_0=0$.

    \item
For the associative case,
 let $\bar V^{\otimes m}$
denote all associative tensor powers of~$V$ over~$C$, written without parentheses, and
$\bar T(V): =  \bigoplus_{m\ge 1} \bar V^{\otimes m}$, with multiplication defined by
$h_1h_2 = h_1 \otimes h_2,$  for monomials $h_1$ and $h_2$. $\bar T(V)$ is isomorphic to the \textbf{free associative algebra} over a basis of $V.$
\end{enumerate}\end{example}

We  also want to make such a construction with vector space pairs. Let us consider $(V,V_0)=\bigoplus _{i\in I} (C,C_0)\cdot x_i $ be the free $(C,C_0)$-module, with
basis $B=\{x_i: i\in I\},$ cf.~\Eref{univex}.

\begin{remark}\label{adm1}$ $
    \begin{enumerate}\eroman
 \item      We could take $C_0=0$ if we want.

         \item
         In the other direction given a free $C$-module $V$,
we could pass to $\tilde V$ and $\tilde C$, to reduce to the case that $(-)\one \in C.$
    \end{enumerate}
\end{remark}

\begin{example}\label{tpow}  Let $(C,C_0)$ be a pair,
 and take $(T(V),T(V)_0 )$  to be the tensor d-bimagma of \Eref{tens}. \begin{enumerate}
    \item (The free $\mcL_0$-additive Lie pair) We take $\mcL$ to be $T(V),$ and $\mcL(V)_0 $ to be the $C$-module generated by $T(V)_0$
and all expressions
\begin{enumerate}
    \item $(\mathbf{x}\mathbf{x})$,
     \item $(\mathbf{x}\mathbf{y}+\mathbf{y}\mathbf{x}),$
   \item $(\mathbf{x}\mathbf{y})\mathbf{z}+(\mathbf{y}\mathbf{z})\mathbf{x}  +(\mathbf{z}\mathbf{x})\mathbf{y},      $
\end{enumerate}
where the $\mathbf{x},\mathbf{y},\mathbf{z}$ are simple tensors.
In view of \Lref{ml}, the axioms of  \Dref{LieD}(i) are satisfied by $(\mcL \mcL_0).$

  \item
If one is willing to modify $C,$ we can define the free Lie pair with basis indexed by any set $I.$ Namely, we take commuting associative indeterminates
$c_{i,j}^k$ over $C,$ and use Lemma~\ref{fgen} to define $\mcL$ over
$C[c_{i,j}^k]$, and formally defining $C_0$ to be the ideal defined by conditions (1)-(4)  of Lemma~\ref{fgen}.

    \item  When $(-)\one \in C$ (which can be attained using Remark~\ref{adm1}), $(T(V),T(V)_0 )^{-}$ is a Lie pair
    by means of Corollary~\ref{Pois1}.
\end{enumerate}
  \end{example}

But we need a surpassing relation $\preceq$ to work with the $\preceq$-adjoint algebra, so we also take a more intricate construction modeled on Proposition~\ref{precvar}.

\begin{example}[The free bilinear $\preceq$-Lie pair]\label{tpowp}
   Re-indexing the subscripts of the $y_i$, we adjoin a formal indeterminate $y_{h_1,h_2,h_3}$
for each $3$-tuple of simple tensors.
We take $\mathcal C$ to be the congruence generated by all pairs
 $$ (h_1 \otimes (h_2\otimes h_3) +y_{h_1,h_2,h_3} ,    h_2 \otimes (h_3 \otimes h_1) + \, h_3\otimes (h_1\otimes h_2))  \,$$ and let $\mathcal U= T(V)/\mathcal C$; i.e., we declare that
 $$ h_1 \otimes (h_2\otimes h_3) +y_{h_1,h_2,h_3} =   h_2 \otimes (h_3 \otimes h_1) +\, h_3 \otimes(h_1\otimes h_2)$$

Let $\mathcal U_0$ be the multiplicative ideal of $ \mathcal U$ generated by all terms
  $$h_i\otimes h_i, \quad h_i \otimes h_j + h_j\otimes h_i,\quad  y_{h_1,h_2,h_3}, \quad i,j,k\{1,2,3\} $$
  where $h_i$ are monomials.
  \end{example}
\begin{theorem}\label{freeL}
$(\mathcal U,\mathcal U_0))$  is a   $\preceq_0$-Lie pair.
Furthermore if $(\mcL ,\mcL _0)$ is a $\preceq_0$-Lie pair then for any $a_i$ in $\mcL$, $i\in I,$ there is a  Lie homomorphism $(\mcL,\mcL_0)\to (\mcL ,\mcL _0)$ sending
$x_i \to \bar x_i: = a_i$ and $ y_{h_1,h_2,h_3}$ to an element $\bar y_{h_1,h_2,h_3}$ of $\mcL _0$ for which $$[\bar h_1 [\bar h_2 \bar h_3]] +\bar y_{h_1,h_2,h_3} = [[\bar h_2 \bar h_3] \bar h_1] \, + [[ \bar h_3 [\bar h_1 \bar h_2].$$
\end{theorem}
\begin{proof}
All the relations except the Jacobi  $\preceq$-identity can be written as identities just in terms of $\mcL$ and $\mcL_0$, so are preserved under substitution. The only difficulty is the Jacobi  $\preceq$-identity, which as in Proposition~\ref{precvar} we rewrite as an identity by inserting the extra term from $\mcL_0$.
(We did not claim uniqueness, since several terms of $\mcL_0$ might provide equality.)

\end{proof}

\begin{remark}
There is a natural map from the   degree 2 part of the exterior semialgebra as in \cite{GaR} to the Lie pair of Example~\ref{cr2}. In fact   we can construct a congruence  of  $\Lcal\otimes\Lcal$  which provides the map to $\Lcal$.
\end{remark}

\subsection{Lie sub-pairs}

\begin{definition}\label{inc}$ $\begin{enumerate}\eroman
 \item    A \textbf{weak $\psi$-Lie sub-pair} of a bimagma pair $ (\mcA,\mcA_0)$ with a pre-negation map $\psi$ is a  sub-pair  $(\mcL,\mcL_0)$,  together with   and a   map $[\phantom{w}]_\psi : \mcL \times \mcL \to \mcA$  satisfying the Lie bracket axioms of \Dref{LieD}(i), as well as the condition
  $$ b_1 b_2 +\psi (b_2b_1) + [b_2 b_1]_\psi \in \mcL_0, \quad \mbox{for all} \  b_1,b_2\in \mcL.$$

    \item A $\preceq$-\textbf{Lie sub-pair} of a   bimagma pair    $ (\mcA,\mcA_0)$  with a surpassing map is a  sub-pair  $ (\mcL,\mcL_0)$,  together with a bilinear map $[\phantom{w}]: \mcL \times \mcL \to \mcA$   satisfying the Lie bracket axioms of \Dref{LieD}(i), as well as the condition $$ b_1 b_2 \preceq_0 b_2b_1 + [ b_1 b_2 ], \quad \mbox{for all} \ b_1,b_2\in \mcL.$$

  \item
  An \textbf{$\psi$-Lie sub-pair} of a bimagma pair $ (\mcA,\mcA_0)$ is a  sub-pair  $ (\mcL,\mcL_0)$,  together with a   map $[\phantom{w}]: \mcL \times \mcL \to \mcA$  satisfying the Lie bracket axioms of \Dref{LieD}(i), as well as the condition
  $$ [b_1b_2] = b_1 b_2  + \psi( b_2) b_1, \quad \mbox{for all} \ b_1,b_2\in \mcL.$$

    {\it We shall call  $[\phantom{w}]$ a bracket, even though we do not require $\mcL $ to be closed under~$[\phantom{w}]$.}
\end{enumerate}
\end{definition}

\begin{lemma}\label{p3}  $ $  \begin{enumerate}\eroman    \item  Any $\psi$-Lie sub-pair  is a \weakLie\
 pair.
   \item For any pre-negation map $\psi$, the bimagma pair  $ (\mcL,\mcL_0)$ is an
 $[\phantom{w}]_\psi $ $\preceq$-sub-pair of itself.
\end{enumerate} \end{lemma}
 \begin{proof} (i)\ \begin{equation}
     \begin{aligned}
      b_1 b_2 +\psi (b_2b_1) + [ b_2 b_1 ] & =  b_1 b_2 +\psi(b_2b_1) + b_2 b_1 +\psi (b_1b_2) \\ &= b_1 b_2   + b_2 b_1 +\psi( b_1 b_2   + b_2 b_1 )\in \mcL_0.
     \end{aligned}
 \end{equation}

       (ii) $ b_1 b_2 \preceq_0 b_1 b_2 + b_2b_1 +\psi(b_2b_1)= b_2b_1 + [ b_1 b_2 ]_\psi $.
       \end{proof}

   \begin{importantnote}\label{P4}
    Lemma~\ref{p3} is applicable quite generally, since  one can pass to the doubled bimagma pair and even take $\psi$ to be multiplication by $(\zero,\one)$. An  instance where one needs to take \weakLie\ sub-pairs: We want to view the free Lie pair inside the free negated associative pair $(T(\Lcal),T(\Lcal)_0)$, which we obtain by doubling in \Rref{adm1}. If we send $x \mapsto (x,0)$, then $[xy] \mapsto ([xy],0)$ whereas $[(x,0),(y,0)] = (xy,yx),$
    which is different. By adjoining all elements of the form $(xy +[yx],yx)$
    and $(xy, [xy]+yx)$ to $T(\Lcal)_0$ we have a \weakLie\ sub-pair.
   \end{importantnote}

\subsubsection{The weak adjoint   morphism}

  Following classical Lie theory, we want to represent Lie pairs inside semiring pairs. The following observation is easy.

  \begin{proposition}[{\cite[Proposition~10.6]{Row16}}] For any $\psi$-Lie pair $(\mcL,\mcL_0)$, there  are weak Lie morphisms
  $\ad: (\mcL,\mcL_0) \to \End (\mcL, \mcL_0)_\psi$, given by $b \mapsto \ad_b$, and, for $\psi$ invertible,  $\ad^\dag: (\mcL,\mcL_0) \to (\End \mcL,\End \mcL_0)_{\psi^{-1}}$, given by $b \mapsto \psi \ad^\dag_b$.
\end{proposition}
\begin{proof}
We verify the conditions of Definition~\ref{morp}. (i) and (ii) are immediate, and (iii) follows from the Jacobi $\mcL_0$-identity. \end{proof}

  Clearly $(\AD_{\mcL},\AD_{{\mcL}_0})$ is a pair. We would like to say that it is a Lie pair under the obvious candidate for Lie bracket, namely $[\ad_x \ad_y]: = \ad_x\ad_y + \ad^\dag_y \ad_x,$ but unfortunately this need not be closed.

\subsection{PBW Theorems for Lie pairs}\label{subsec:PBWy}

Throughout this section  suppose that $(\mcL,\mcL_0)$ is a  Lie pair, where $\mcL$ is also a free $C$-module with basis $\{x^i: i \in J\},$ and $\mcL_0$ is the submodule with basis $\{x^i: i \in J_0 \subset J\}$.
Reversing the direction of Theorem~\ref{pLie}, we want a universal enveloping construction of a  semiring pair from the Lie pair $(\Lcal, \Lcal_0)$. In classical theory this is the celebrated PBW (Poincare-Birkhoff-Witt) Theorem.

For Lie pairs there are three possible versions $\iota: (\Lcal, \Lcal_0) \to U$ where $U$ is respectively $U_{\psi}(\Lcal, \Lcal_0),$ $U_\preceq(\Lcal, \Lcal_0)$, $U_\vep(\Lcal, \Lcal_0),$ depending on which type of Lie pair and which type of morphism $\iota$ we use (resp.~ weak $\vep$-Lie morphism,  $\preceq$-Lie morphism, $\vep$-Lie homomorphism), which we fix in the next definition.

\begin{definition}
   {\bf Universal   Property.} If $(\Acal, \Acal_0)$ is any associative pair given together with a morphism $f:(\Lcal,\Lcal_0)\to (\Acal,\Acal_0)$ such that $f$ satisfies resp.~(i), (ii), (iii) of \Dref{inc}, then there is a   unique respective morphism $\phi_f:\Ucal(\Lcal, \Lcal_0)\to (\Acal,\Acal_0)$ such that $f=\phi_f\circ \iota$.
\end{definition}

Note that we did not require $\iota$ to be injective; this will be examined each time.
The reduction techniques used in classical Lie theory become unusable without  cancellation, but in the semialgebra case we can often apply a degree argument to the elements of the tensor algebra in the following situation, since we only adjoin monomials of degree $\ge 2$ in the $x^i$ to $\mcA_0.$

\begin{definition}\label{LZS} A semigroup
  $(\mcA,\zero)$ satisfies the    \textbf{lacks zero sums} (LZS) property if the sum of nonzero elements of $\mcA$ cannot equal $\zero.$
\end{definition}

The LZS property will be the key to obtaining an injection in Theorems~\ref{PBW0} and \ref{thm:PBW}.
   \subsubsection{The weak $\psi$-version of PBW}

\begin{theorem}\label{PBW0} Suppose $(\mcL,\mcL_0)$  is a $\psi$-Lie pair.
 Let $V = \mcL,$ extended by a formal set of indeterminates $Y = \{ y_{i,j}: i,j\in J\}.$ Define $U_{\wk;\psi}(\mcL)=T(V)$ using the construction of \Eref{tpow}, and, identifying $x $ with $\iota(x)$ for $x$ in $\mcL,$ let
$U_{\wk;\psi}(\mcL)_0$ be the $C$-submodule generated by $\mathcal L(V)_0$ and  $$\{  x^ix^j + \psi (x^j)   x^i + [x^j x^i] : i,j \in I  \}.$$ Define $U_{\wk; \psi}(\mcL, \mcL_0) := (U_{\wk;\psi}(\mcL), U_{\wk;\psi}(\mcL)_0). $ It is worth noticing that only the null part depends on $\psi.$

\begin{enumerate}

    \item  $U_{\wk; \psi}(\mcL, \mcL_0) $  is a   $\psi$-Lie pair, as in Theorem~\ref{pLie}.

    \item There is a universal weak $\psi$-Lie morphism   $\iota_ \psi \:(\mcL,\mcL_0)\to U_{\wk; \psi}(\mcL, \mcL_0) $ given by $ x^i \mapsto   x^i,$   satisfying the Universal  Property in this setting.

  \item The universal  $\iota_ \psi$ is $\mcL_0$-injective when $\mcL$  satisfies LZS.
\end{enumerate}
\end{theorem}
\begin{proof} By definition $\iota(\mcL_0)\subseteq U_{\wk;\psi}(\mcL)=T(V).$ Also
$$[x^i ,x^j]_\psi + [x^j x^i] =     x^ix^j + \psi (x^j)   x^i + [x^j x^i]  \in \mcA_0     $$
by definition, so $\iota$ is a weak Lie morphism.
Uniqueness is clear since $\varphi$ must satisfy $\varphi(\iota(x^i)) = f(x_i)$.

It remains to prove that  $\iota$ is $\mcL_0$-injective when $\mcL$  satisfies LZS. This is seen
seen by checking degrees in the tensor semialgebra. Namely,
the degree 1 cannot be in $\mcL_0$ because of the LZS  Property. (Here the lack of negation makes life easier, because there is no ambiguity!)

\end{proof}

\subsubsection{The $\preceq$ version of PBW}

Now, given a Lie pair $(\Lcal,\Lcal_0)$  endowed with a surpassing map $\preceq$, we want to construct an associative negated pair $(\Ucal_\preceq(\Lcal,\Lcal_0))$ such that there is a universal $\preceq$-embedding $\iota$  of $(\Lcal,\Lcal_0)$ such that \be xy\preceq \iota([xy])+yx,
\ee
satisfying the $\preceq$-universal  Property.

There exists a map $\phi:\Ucal(\Lcal, \Lcal_0)\to (\Acal,\Acal_0)$ such that $f=\phi\circ \iota$.

This is a bit subtler than before.
  Bergman~\cite{Ber} found a beautiful method of proving the PBW Theorem, related to Gr\"obner bases, to determine bases of algebras, but lacking negation
is both a hindrance and an asset, as we shall see.

\begin{theorem}\label{thm:PBW}
Suppose $(\mcL,\mcL_0)$ is a  $\preceq$-Lie pair satisfying LZS, where $\mcL$ is also a free $C$-module with basis $\{x^i: i \in I\},$ where we order the index set $I$, and
$\mcL_0$ is the submodule with basis $\{x^i: i \in J_0 \subset
J\}$. We refine \Eref{tpow}. Define $T(V)_>$ to be the subspace of $T(V)$ spanned by monomials
$x_{\mathbf{i}}:= x^{i_1}\otimes\cdots\otimes x^{i_m}$, where $i_1> \dots >i_m.$

Let $V = \mcL.$ We take $W_0=\{ y_{j,i}:=   : i<j \in I\}$, and $U_{ \preceq}(\mcL)$  the semialgebra freely generated by $T(V)_>$ and $W_0$, modulo the relations in the congruence generated by the relations $x^j x^i +y_{j,i} = x^ix^j + [x^j x^i] , $ for all $j>i$, and $U_{ \preceq}(\mcL)_0$ the multiplicative ideal of $U_{ \preceq}(\mcL, \mcL_0)$ generated by $W_0$ and $\mcL_0$.
Then
\begin{enumerate}
    \item  $U_{ \preceq}(\mcL;\mcL_0) :=(U_{ \preceq}(\mcL),U_{ \preceq}(\mcL))_0$ defines a $\preceq$-Lie pair,
and there is a universal $\preceq$-morphism $\varphi:U_{ \preceq}(\mcL, \mcL_0)\to (\tilde {\mcL}, \tilde {\mcL}_0)$ given by $ x^i \mapsto \bar x^i.$

   \item Furthermore,  $\iota\:(\mcL,\mcL_0)\to (U_{ \preceq}(\mcL, \mcL_0))$ is $\mcL_0$-injective    when $\mcL$  satisfies LZS.
\end{enumerate}
\end{theorem}
\begin{proof}
First we note that $\iota$ is a $\preceq$-homomorphism. By definition
\[
x^ix^j + y_{j,i} = x^j x^i +y_{j,i} \succeq x^j x^i.
\]
This extends to the congruence.

To prove that $\iota$  is ${\mcL}_0$-injective  when $\mcL$  satisfies LZS, we simply note that all the relations have degree $\ge 2$ in the $x^i$, so they intersect trivially with $\mcL_0.$
\end{proof}

\begin{remark}
    What can be said when $\mcL$ does not satisfy LZS? Since the Jordan algebraic version of the PBW fails, we must
 deal with the ambiguities using the Lie product. Any  ambiguity involves rearranging   sequences of $x^i$ into ascending sequences. But the parts of $T(V)_>$ match and these are stipulated to be canceled, so we are left with relations in $\mcA_0;$ for example for $i<j<k$ one  considers $x^k (x^j x^i) $ versus $(x^ k x^j)x^i $, which is resolved by rearranging them and  canceling $x^i x^j x^k:$
\begin{eqnarray}
x^k (x^j x^i) &=& x^k (x^i x^j + [x^jx^i] +y_{j,i}) =  x^i x^k x^j + y_{k,i}x^j +x^k y_{j,i} +x^k  [x^jx^i]  \cr  &=& x^ix^jx^k +(x^i y_{k,j}  + y_{k,i}x_j +x^k y_{j,i} +x^k  [x^jx^i]).
\end{eqnarray}

Canceling out $x^j x^i x^k $  yields a relation holding in any classical Lie algebra, so we need some further cancellative  property to be in a position to apply the classical PBW theorem.
\end{remark}

   \subsubsection{The $\vep$-version of PBW}

Suppose that $(\mcL,\mcL_0)$ is a $\vep$-Lie pair, with $\vep\in C. $\footnote{We could work more generally with a pre-negation map $\psi$ on $\mcL$  if we mod $T(\mcL)$ by the congruence generated by   $(\psi(x)\otimes y,x \otimes \psi(y))$ for all $x,y\in \mcL$.}
Recall that
$
[x,y]_\epsilon=xy+\epsilon yx.$
Then
\be
[x,y]_\epsilon+[y,x]_\epsilon\in T(\Lcal_0):=(1+\epsilon)T(\Lcal) \label{eq:ep2}
\ee

$$
xy+yx(1+\epsilon)=xy+\epsilon yx+yx=yx+[x,y]_\epsilon\quad \mbox{for all} \ x,y\in T(\Lcal)
$$
i.e.
\be
xy\preceq_0  yx+[x,y]_\epsilon\label{eq.preceq}\quad \mbox{for all} \ x,y\in T(\Lcal)
\ee
There is a natural injection $\iota: (\Lcal;\Lcal_0)\lra (T(\Lcal),T(\Lcal_0))$.

Define now $\Ucal_\vep(\Lcal)$ to be  $T(\Lcal)$ modulo the congruence $\Cong$
 generated by $$([x_ix_j]_\vep , x_ix_j + \vep x_jx_i)$$ for   elements in $\Lcal=T^1(\Lcal)$. In other words, if $x_i,x_j\in  \Lcal$ then $\iota(x_i)\iota(x_j)+\vep \iota(x_j)\iota(x_i)$ is identified with    the $\iota$ image of $\epsilon[x_ix_j]\in \Lcal$ in $T^1(\Lcal)\subseteq T(\Lcal)$. Similarly let $\Ucal(\Lcal_0)=(1+\epsilon)\Ucal$.

\begin{theorem}\label{PBW3} $ $\begin{enumerate}
     \item  $\Ucal_{\vep}(\mcL;\mcL_0) := (\Ucal_{\vep}(\mcL),\Ucal_{\vep}(\mcL)_0)$ defines a $\vep$-Lie pair, which is strong when $(\mcL;\mcL_0)$ is a strong $\vep$-Lie pair.

 \item There is a universal $\vep$-Lie homorphism $\varphi:\Ucal_{\vep}(\mcL, \mcL_0)\to (\tilde {\mcL}, \tilde {\mcL}_0)$ given by $ x^i \mapsto \bar x^i.$

\item Let $(\Acal,\Acal_0)$ be any associative $(\Ccal;\Ccal_0)$-semiring pair, and let $f:(\Lcal,\Lcal_0)\to (\Acal, \Acal_0)$ be any map such that

 \begin{equation}\label{fix}
     f(x)f(y)+\vep f(y)f(x)=f([x,y]).
 \end{equation}
 Then there is a unique homomorphism $\psi_f:(\Ucal_\vep(\Lcal),\Ucal_\vep(\Lcal_0))\to (\Acal, \Acal_0)$ such that $f=\psi_f\circ \iota$.
\end{enumerate}
\end{theorem}
\begin{proof} (1) and (2) are as in the proofs of Theorems \ref{PBW0} and \ref{thm:PBW}, using Theorem~\ref{pLie}.

(3) We define the map $T(\Lcal)\to \Acal$ given by $x_i\mapsto f(x_i)$. By \eqref{fix} this map factors through $\Cong,$ yielding the desired homomorphism $\psi_f:U(\Lcal)\to \mcA.$
Moreover we also have $\psi_f(U(\Lcal)) \subseteq (\one+\vep) \mcA =  \mcA_0.$
\end{proof}
\medskip
\noindent
{\bf Description of $\Ucal_\vep(\Lcal)$}. As a $\Ccal$-module, $\Ucal_\vep(\Lcal)$ is spanned by finite linear combinations of monomials $x_1^{i_1}\cdots x_k^{i_k}$, using \eqref{fix} to reduce whenever possible.

We can obtain surpassing reductions, as in the following example:
\begin{example}$ $
\begin{enumerate}
\item Let $x,y,z,\in \Lcal$. Then
\begin{align*}
zxy&\preceq (xz+\iota([zx]))y=xzy+\iota([zx])y\\
&\preceq x(yz+(\iota[z,y])+\iota([zx])y=xyz+x\iota([zx])+\iota([zx])y
\end{align*}
In other words
$$
zxy\preceq xyz+x\iota([zx])+\iota([zx])y.
$$
Notice that the right hand side  only involves product of brackets in $\Lcal$ and of product of $x$, $y$, $z$ in alphabetical order. We can re-arrange the factors, paying the price of adding elements of lower degree.
However, we may have extra terms of degree 1, so we may not have an injection.
\end{enumerate}
\end{example}

\begin{remark}
    The identity map on $\Lcal$ induces a weak Lie morphism from  $ U_{\wk}(\mcL)=T(V)$ to  $\Ucal_\vep(\Lcal),$ extending the identity on $\Lcal.$
\end{remark}

\section*{Acknowledgments}

The authors thank the referee for careful readings, and for sound advice on improving the presentation. The first author was supported partially by INDAM-GNSAGA, PRIN Multilinear Algebraic Geometry, and RIB23GATLET. The second author was supported by the Israel Science Foundation grant 1994/20 and the Anshel Pfeffer Chair.


\begin{thebibliography}{99}

\bibitem{AGR1} M. Akian, S. Gaubert, and L. Rowen. Semiring systems arising from hyperrings, arXiv 2207.06739, to appear in Journal of Pure and Applied Algebra, 2023.

\bibitem{AGR2} M. Akian, S. Gaubert, and L. Rowen. Linear algebra over pairs, arXiv 2310.05257, 2023.

\bibitem{BaGoRem} Y. Bahturin, M. Goze, and E. Remm. Group gradings on filiform Lie algebras. {\it Comm. Algebra}, 44(1):40--62, 2016.

\bibitem{Bai} C. Bai. {\it An introduction to pre-lie algebras}. Algebra and Applications. Wiley, 2021.

\bibitem{Ber} G. M. Bergman. The Diamond Lemma for Ring Theory. {\it Advances in Mathematics}, 29:178--218, 1978.

\bibitem{Bl} G. Blachar. Modules and Lie semialgebras over semirings with a Negation Map, 2017.

\bibitem{CK} C. Chang and H. Keisler. {\it Model Theory}. Chapter 6. Studies in Logic and the Foundations of Mathematics (3rd ed.). Elsevier, 1990.

\bibitem{ChaGaRo} A. Chapman, L. Gatto, and L. Rowen. Clifford Semialgebras. {\it Rend. Circ. Mat. Palermo}
(2), 72(2):1197--1238, 2023.

\bibitem{Cos} A. Costa. Sur la th\^{e}orie g\'{e}n\'{e}rale des demi-anneaux. {\it Publ. Math. Decebren}, 10:14--29, 1963.

\bibitem{DJKM} E. Date, M. Jimbo, M. Kashiwara, and T. Miwa. Transformation groups for soliton equa-
tions. III. Operator approach to the Kadomtsev–Petviashvili equation. {\it J. Phys. Soc. Jpn}, 50(11):3806--3812, 1981.

\bibitem{GaR} L. Gatto and L. Rowen. Grassman semialgebras and the Cayley-Hamilton theorem. {\it Proc. Amer. Math. Soc. Ser. B}, 7:183--201, 2020.

\bibitem{golan92} J. Golan. {\it Semirings and their Applications}. Springer-Science + Business, Dordrecht, 1999.

\bibitem{HH} J. Hilgert and K. Hofmann. Semigroups in Lie groups, semialgebras in Lie algebras. {\it Trans. Amer. Math. Soc.}, 288(2):481--504, 1985.

\bibitem{H} Humphreys. {\it Introduction to Lie Algebras and Representation Theory}. Springer Graduate
Texts in Mathematics (GTM, volume 9), 1972.

\bibitem{Jac} N. Jacobson. {\it Lie algebras}. Dover Publications, Inc., New York, 1979.

\bibitem{Jac1980} N. Jacobson. {\it Basic algebra II}. Freeman, 1980.

\bibitem{JMR} J. Jun, K. Mincheva, and L. Rowen. $\mathcal{T}$-semiring pairs. {\it Volume in honour of Prof. Martin Gavalec Kybernetika}, 58:733--759, 2022.

\bibitem{JuR1} J. Jun and L. Rowen. {\it Categories with negation}. In Categorical, homological and combinatorial methods in algebra. Contemp. Math., {\bf 751}, Amer. Math. Soc., [Providence], RI, 2020.

\bibitem{Ka1} Y. Katsov. Tensor products and injective envelopes of semimodules over additively regular
semirings. {\it Algebra Colloquium}, 4(2):121--131, 1997.

\bibitem{krasner} M. Krasner. A class of hyperrings and hyperfields. {\it Internat. J. Math. and Math. Sci.}, 6(2):307--312, 1983.

\bibitem{Row17} L. Rowen. {\it An informal overview of triples and systems}. Rings, modules and codes. Contemp. Math., {\bf 727}, Amer. Math. Soc., [Providence], RI, 2019.

\bibitem{Row16} L. Rowen. Algebras with a negation map. {\it European Journal of Mathematics}, 8:62--138, 2022.

\bibitem{Vir} O. Viro. Hyperfields for Tropical Geometry I. Hyperfields and dequantization, 2010, arXiv:1006.3034.

\end{thebibliography}

\EditInfo{October 13, 2023}{December 11, 2023}{Ivan Kaygorodov, Adam Chapman, Mohamed Elhamdadi}

\end{document}